\documentclass[a4paper,oneside,english]{amsart}
\usepackage[T1]{fontenc}
\usepackage[latin9]{inputenc}
\usepackage{refstyle}
\usepackage{amstext}
\usepackage{amsthm}
\usepackage{amssymb}
\usepackage{setspace}

\makeatletter

\AtBeginDocument{\providecommand\thmref[1]{\ref{thm:#1}}}
\AtBeginDocument{\providecommand\partref[1]{\ref{part:#1}}}
\AtBeginDocument{\providecommand\lemref[1]{\ref{lem:#1}}}
\AtBeginDocument{\providecommand\enuref[1]{\ref{enu:#1}}}
\AtBeginDocument{\providecommand\defref[1]{\ref{def:#1}}}

\RS@ifundefined{subsecref}
  {\newref{subsec}{name = \RSsectxt}}
  {}
\RS@ifundefined{thmref}
  {\def\RSthmtxt{theorem~}\newref{thm}{name = \RSthmtxt}}
  {}
\RS@ifundefined{lemref}
  {\def\RSlemtxt{lemma~}\newref{lem}{name = \RSlemtxt}}
  {}

\numberwithin{equation}{section}
\numberwithin{figure}{section}
\theoremstyle{plain}
\newtheorem{thm}{\protect\theoremname}[section]
  \theoremstyle{remark}
  \newtheorem{rem}[thm]{\protect\remarkname}
  \theoremstyle{plain}
  \newtheorem{lem}[thm]{\protect\lemmaname}
  \theoremstyle{definition}
  \newtheorem{defn}[thm]{\protect\definitionname}
  \theoremstyle{definition}
  \newtheorem{example}[thm]{\protect\examplename}

\usepackage{etoolbox}

\patchcmd{\subsection}{-.5em}{.5em}{}{}

\makeatother

\usepackage{babel}
  \providecommand{\definitionname}{Definition}
  \providecommand{\examplename}{Example}
  \providecommand{\lemmaname}{Lemma}
  \providecommand{\remarkname}{Remark}
\providecommand{\theoremname}{Theorem}

\begin{document}
\begin{doublespace}

\title[]{On the Woodin Construction of Failure of GCH at a Measurable Cardinal}
\end{doublespace}

\author{Yoav Ben Shalom }
\begin{abstract}
{\normalsize{}Let GCH hold and let $j:V\longrightarrow M$ be a definable
elementary embedding such that $crit(j)=\kappa$, $^{\kappa}M\subseteq M$
and $\kappa^{++}=\kappa_{M}^{++}$. H. Woodin (see \cite{key-1})
proved that there is a cofinality preserving generic extension in
which $\kappa$ is measurable and GCH fails at it. This is done by
using an Easton support iteration of Cohen forcings for blowing the
power of every inaccessible $\alpha\leq\kappa$ to $\alpha^{++}$,
and then adding another forcing on top of that. We show that it is
enough to use the iterated forcing, and that the latter forcing is
not needed. We will show this not only for the case where $\kappa^{++}=\kappa_{M}^{++}$,
but for every successor ordinal $\gamma$, where $0<\gamma<\kappa$,
we will show it when the assumption is $\kappa^{+\gamma}=\kappa_{M}^{+\gamma}$. }{\normalsize \par}
\end{abstract}

\maketitle

\section{Introduction}

H. Woodin (see \cite{key-1}) proved the following:
\begin{thm}
\label{thm:Woodin-theorem}Let GCH hold and let $j:V\longrightarrow M$
be a definable elementary embedding such that $crit(j)=\kappa$, $^{\kappa}M\subseteq M$
and $\kappa^{++}=\kappa_{M}^{++}$. Then there is a cofinality preserving
generic extension in which $\kappa$ is measurable and GCH fails at
it.
\end{thm}
\begin{rem}
The hypotheses of \thmref{Woodin-theorem} can easily be had from
a cardinal $\kappa$ which is $(\kappa+2)$-strong. Work of Gitik
\cite{key-2} shows that they can be forced starting with a model
of $o(\kappa)=\kappa^{++}$, and by work of Mitchell \cite{key-3}
this is optimal.
\end{rem}
The proof starts by defining an iteration forcing $\mathbb{P}^{2}=\mathbb{P}_{\kappa+1}^{2}$
with Easton support, where for $\alpha\leq\kappa$ we let $\mathbb{Q}_{\alpha}=Add(\alpha,\alpha^{++})_{V[G_{\alpha}]}$
for inaccessible $\alpha$, and the trivial forcing otherwise. Then
we lift the embedding $j:V\longrightarrow M$ to $V[G_{\kappa+1}]$
(and an appropriate extension of $M$). In order to do so an additional
forcing is used on top of $\mathbb{P}$, which is equivalent to $Add(\kappa^{+},\kappa^{++})$.
In \partref{2} we will show it is possible to construct the embedding
without any further forcing, i.e. we will prove the following:
\begin{thm}
\label{thm:2}Let GCH hold and let $j:V\longrightarrow M$ be a definable
elementary embedding such that $crit(j)=\kappa$, $^{\kappa}M\subseteq M$
and $\kappa^{++}=\kappa_{M}^{++}$. Then it is possible to force $V$
by $\mathbb{P}^{2}$ and lift $j$ to the extension while preserving
definability already in $V^{\mathbb{P}^{2}}$.
\end{thm}
The arguments of \thmref{2} can be generalized, so let $\mathbb{P}^{\gamma}=\mathbb{P}_{\kappa+1}^{\gamma}$
be an iteration forcing with Easton support, where for $\alpha\leq\kappa$
we let $\mathbb{Q}_{\alpha}=Add(\alpha,\alpha^{+\gamma})_{V[G_{\alpha}]}$
for inaccessible $\alpha$, and the trivial forcing otherwise. In
\partref{3} we prove that:
\begin{thm}
Let GCH hold and let $j:V\longrightarrow M$ be a definable elementary
embedding such that $crit(j)=\kappa$, $^{\kappa}M\subseteq M$ and
$\kappa^{+3}=\kappa_{M}^{+3}$. Then it is possible to force $V$
by $\mathbb{P}^{3}$ and lift $j$ to the extension while preserving
definability already in $V^{\mathbb{P}^{3}}$.
\end{thm}
Finally, in \partref{4} we generalize for every successor ordinal
$\gamma$, where $0<\gamma<\kappa$:
\begin{thm}
Let GCH hold and let $j:V\longrightarrow M$ be a definable elementary
embedding such that $crit(j)=\kappa$, $^{\kappa}M\subseteq M$ and
$\kappa^{+\gamma}=\kappa_{M}^{+\gamma}$. Then it is possible to force
$V$ by $\mathbb{P}^{\gamma}$ and lift $j$ to the extension while
preserving definability already in $V^{\mathbb{P}^{\gamma}}$.
\end{thm}

\section*{Acknowledgments}

I would like to express my sincere gratitude to Prof. Moti Gitik for
his guidance, insights and encouragement throughout this research.
Without him this would not have been possible.

\section{Preliminaries}

\subsection{Ultrapowers and Extenders}

We will start with the same constructions as in Woodin's arguments.
For $a\in[\kappa^{+\gamma}]^{<\omega}$ Define $U_{a}=\{X\subseteq[\kappa]^{|a|}:a\in j(X)\}$
and let $i_{a}$ be the ultrapower map from $V$ to $N_{a}\simeq Ult(V,U_{a})$.
Notice that if we let $k_{a}:N\longrightarrow M$ be $k_{a}([F]_{U_{a}})=j(F)(a)$
then $j=k_{a}\circ i_{a}$.
\begin{lem}
$k$ is an elementary embedding.
\end{lem}
\begin{proof}
Let $\phi(x_{1},\ldots,x_{n})$ be a formula and $[f_{1}]_{U_{a}},\ldots,[f_{n}]_{U_{a}}\in N_{a}$
parameters such that $N_{a}\models\phi([f_{1}]_{U_{a}},\ldots,[f_{n}]_{U_{a}})$.
Then from Los's theorem $\{x\in[\kappa]^{|a|}|V\models\phi(f_{1}(x),\ldots,f_{n}(x))\}\in U_{a}$.
From the definition of $U_{a}$ this means that:
\begin{align*}
a & \in j(\{x\in[\kappa]^{|a|}|V\models\phi(f_{1}(x),\ldots,f_{n}(x))\})=\\
 & =\{x\in j([\kappa]^{|a|})|M\models\phi(j(f_{1})(x),\ldots,j(f_{n})(x))\}
\end{align*}
 So if $a$ is in this set, we can substitute it as $x$ and get $M\models\phi(j(f_{1})(a),\ldots,j(f_{n})(a))=\phi(k([f_{1}]_{U_{a}}),\ldots,k([f_{n}]_{U_{a}}))$.
\end{proof}
\begin{lem}
\label{lem:unify-indexes}For every $a_{1},\ldots,a_{n}\in[\kappa^{+\gamma}]^{<\omega}$
there is some $b\in[\kappa^{+\gamma}]^{<\omega}$ such that $\forall i\in[n]:k_{a_{i}}^{''}N_{a_{i}}\subseteq k_{b}^{''}N_{b}$.
\end{lem}
\begin{proof}
If we can prove this for $n=2$, it is trivial to extend it by induction
for any $n$. So let us assume $n=2$. Take $b=a_{1}\smallfrown a_{2}$,
where $\smallfrown$ is concatenation. For $a_{1}$, assume $[f_{1}]_{U_{a_{1}}}\in N_{a_{1}}$.
Denote $h_{a_{1}}:[\kappa]^{|a_{1}|+|a_{2}|}\longrightarrow[\kappa]^{|a_{1}|}$
to be the projection map which for $x\in[\kappa]^{|a_{1}|+|a_{2}|}$
takes the first $|a_{1}|$ coordinates. Then notice that:
\[
k_{a_{1}}([f_{1}]_{U_{a_{1}}})=j(f_{1})(a_{1})
\]
\[
k_{b}([f_{1}\circ h_{a_{1}}]_{U_{b}})=j(f_{1}\circ h_{a_{1}})(b)=j(f_{1}\circ h_{a_{1}})(a_{1}\smallfrown a_{2})=j(f_{1})(a_{1})
\]
 Which means that $k_{a_{1}}^{''}N_{a_{1}}\subseteq k_{b}^{''}N_{b}$.
The case for $a_{2}$ is similar (taking $h_{a_{2}}$ which takes
the last $|a_{2}|$ coordinates) which concludes the proof.
\end{proof}
\begin{defn}
The definable elementary embedding $j$ is said to be $j_{E}^{V}$
for $E$ a $(\kappa,\kappa^{+\gamma})$-extender if $crit(j)=\kappa$,
$j(\kappa)>\kappa^{+\gamma}$ and every element of $M$ is of the
form $k(F)(a)$ for some $a\in M$ and function $F\in V$, where $a\in[\kappa^{+\gamma}]^{<\omega}$
and $dom(F)=[\kappa]^{|a|}$.
\end{defn}
\begin{lem}
We may assume that $j=j_{E}^{V}$ for some $(\kappa,\kappa^{+\gamma})$-extender
E.
\end{lem}
\begin{proof}
Take $M'=\{j(F)(a)|F\in V,F:\kappa\longrightarrow V,a<\kappa^{+\gamma}\}$.
Notice that $Rng(j)\subseteq M'$. Let us show that $M'$ is an elementary
submodel of $M$. Notice that if we show that, then for every formula
$\phi(x_{1},\ldots,x_{n})$ and parameters $a_{1},\ldots,a_{n}\in V$
we will have $V\models\phi(a_{1},\ldots,a_{n})$ iff $M\models\phi(j(a_{1}),\ldots,j(a_{n}))$
iff $M'\models\phi(j(a_{1}),\ldots,j(a_{n}))$, so it will end the
proof. We will use Tarski-Vaught test, so let $\phi(x,y_{1},\ldots,y_{n})$
be some first-order formula, and $j(f_{1})(b_{1}),\ldots,j(f_{n})(b_{n})\in M'$
parameters such that $M\models\exists x:\phi(x,j(f_{1})(b_{1}),\ldots,j(f_{n})(b_{n}))$.
From \lemref{unify-indexes} we can assume that $b_{1}=\ldots=b_{n}=b$.
So we can re-write the formula as $M\models\exists x:\phi(x,k_{b}([f_{1}]_{U_{b}}),\ldots,k_{b}([f_{n}]_{U_{b}}))$,
and the using elementarity get $N_{b}\models\exists x:\phi(x,[f_{1}]_{U_{b}},\ldots,[f_{n}]_{U_{b}})$.
Denote by $[f]_{U_{b}}$ such $x$. Then $N_{b}\models\exists x:\phi([f]_{U_{b}},[f_{1}]_{U_{b}},\ldots,[f_{n}]_{U_{b}})$,
so from elementarity $M\models\exists x:\phi(k_{b}([f]_{U_{b}}),k_{b}([f_{1}]_{U_{b}}),\ldots,k_{b}([f_{n}]_{U_{b}}))=\phi(j(f)(b),j(f_{1})(b),\ldots,j(f_{n})(b))$.
From the definition $j(f)(b)\in M'$, as we wanted.
\end{proof}

\subsection{Initial Forcing}

Let $\lambda=\kappa_{N}^{+\gamma}$. Let $G$ be a generic filter
for $\mathbb{P}_{\kappa}$ over $V$ and let $g$ be a generic filter
for $\mathbb{Q}_{\kappa}$ over $V[G]$. Since $^{\kappa}M\subseteq M$
and also $^{\kappa}N\subseteq N$ (since it is the ultrapower by a
measure over $\kappa$) the iterations $\mathbb{P}$, $i(\mathbb{P})$,
$j(\mathbb{P})$ agree up to stage $\kappa$ (so we can use $G$ for
the first $\kappa$ iterations in all of them).
\begin{lem}
$j(\mathbb{P})_{\kappa+1}=\mathbb{P}_{\kappa+1}$
\end{lem}
\begin{proof}
We have seen that $\mathbb{P}$, $j(\mathbb{P})$ agree up to stage
$\kappa$, so it is left to see they agree for stage $\kappa$. From
elementarity $\mathbb{Q}_{\kappa}^{M[G]}=Add(\kappa,\kappa^{+\gamma})_{M[G]}$.
But $\kappa_{M}^{+\gamma}=\kappa^{+\gamma}$, and since all the elements
are of cardinality $<\kappa$ (and we have $V\models{}^{\kappa}M\subseteq M$
so from \cite{key-1} (8.4) we have $V[G]\models{}^{\kappa}M[G]\subseteq M[G]$)
this is simply $Add(\kappa,\kappa^{+\gamma})_{V[G]}=\mathbb{Q}_{\kappa}$.
\end{proof}
\begin{lem}
$i(\mathbb{P})_{\kappa+1}=\mathbb{P}_{\kappa}\ast\dot{\mathbb{Q}}_{\kappa}^{*}$,
where $\mathbb{Q}_{\kappa}^{*}=Add(\kappa,\lambda)_{V[G]}$.
\end{lem}
\begin{proof}
Again, We have seen that $\mathbb{P}$, $i(\mathbb{P})$ agree up
to stage $\kappa$, so it is left to check it for stage $\kappa$.
From elementarity $\mathbb{Q}_{\kappa}^{N[G]}=Add(\kappa,\kappa^{+\gamma})_{N[G]}$.
But $\kappa_{N}^{+\gamma}=\lambda$, and since all the elements are
of cardinality $<\kappa$ (and we have $V\models{}^{\kappa}N\subseteq N$
so from \cite{key-1} (8.4) we have $V[G]\models{}^{\kappa}N[G]\subseteq N[G]$)
this is simply $Add(\kappa,\lambda)_{V[G]}$.
\end{proof}
\begin{lem}
$crit(k)\geq\kappa^{+}$.
\end{lem}
\begin{proof}
For $\alpha<\kappa$ it is clear that $k(\alpha)=\alpha$. For $\kappa$
we have $\kappa\leq k(\kappa)\leq k([Id_{\kappa}]_{U})=j(Id_{\kappa})(\kappa)=\kappa$

It is easy to see that $k^{''}G=G$, so we can lift $k$ to get $k:N[G]\longrightarrow M[G]$.
Use \cite{key-1} (15.6) to get $g_{0}$, a generic filter for $\mathbb{Q}_{\kappa}^{*}$
over $N[G]$ such that $k^{''}g_{0}\subseteq g$. This means we can
again lift $k$ to get $k:N[G\ast g_{0}]\longrightarrow M[G\ast g]$
Let $\mathbb{R}_{0}=G_{\kappa+1,i(\kappa)}^{N}$ be the factor forcing
to prolong $G\ast g_{0}$ to a generic filter for $i(\mathbb{P}_{\kappa})$.
Then from \cite{key-1} (8.1) we can build $H_{0}\in V[G\ast g_{0}]$
which is a generic filter for $\mathbb{R}_{0}$ over $N[G\ast g_{0}]$.
\end{proof}
\begin{defn}
The ``width'' of the elementary embedding $k$ is said to be $\leq\mu$
iff every element of $M$ is of the form $k(F)(a)$ for some function
$f\in N$ and $a\in M$, where $N\models|dom(F)|\leq\mu$.
\end{defn}
\begin{lem}
$k$ is of width $\leq\lambda$.
\end{lem}
\begin{proof}
Let $a\in M$. Then there is some function $f$ and $\beta<\kappa^{+\gamma}$
such that $j(f)(\beta)=a$. Notice that $j(f)(\beta)=k(i(f))(\beta)$.
Since $\beta<\kappa^{+\gamma}$ and $k(\kappa_{N}^{+\gamma})=\kappa_{M}^{+\gamma}=\kappa^{+\gamma}$,
$\beta$ will be in the domain of $k(i(f)\upharpoonright\kappa_{N}^{+\gamma})$
so we get $k(i(f)\upharpoonright\lambda)(\beta)=k(i(f)\upharpoonright\kappa_{N}^{+\gamma})(\beta)=a$.
This means that $k$ is of width $\leq\lambda$. Notice that this
stays true for lifts of $k$, since instead of $f$ such that $j(f)(\beta)=a$
we can take it such that $j(f)(\beta)$ is the $M-$name of $a$.
\end{proof}
Using the last lemma we can use \cite{key-1} (15.1) and transfer
$H_{0}$ along $k$ to get $H$.

All together, we get the following commutative triangle:
\[
\begin{array}{ccc}
V[G] & \overset{j}{\longrightarrow} & M[G\ast g\ast H]\\
 & \overset{i}{\searrow} & \uparrow k\\
 &  & N[G\ast g_{0}\ast H_{0}]
\end{array}
\]

\subsection{Building a Generic Filter for $Add(j(\kappa),j(\gamma))^{M[G\ast g\ast H]}$}

In the following parts we will build the the remaining piece for $M[G\ast g\ast H]$
which will correspond to $g$. Notice that given a generic filter
for $Add(j(\kappa),j(\gamma))^{M[G\ast g\ast H]}$ over $M[G\ast g\ast H]$
which is defined in $V[G\ast g]$, we can simply continue with the
proof in \cite{key-1} to get a mapping $j:V[G\ast g]\rightarrow M[G\ast g\ast H\ast f^{*}]$
which is defined in $V[G\ast g]$. So our goal will be to find a generic
filter for $Add(j(\kappa),j(\gamma))^{M[G\ast g\ast H]}$ over $M[G\ast g\ast H]$.
We will build it first for $\gamma=\kappa^{+2}$, then for $\gamma=\kappa^{+3}$,
and finally for every cardinal $\gamma$, $\kappa<\gamma<\kappa^{+\kappa}$.
All the work from here on, unless otherwise specified, will be in
$V[G\ast g]$.

Start by some basic lemmas and definitions:
\begin{lem}
\label{lem:1}Let $\lambda$ be a cardinal such that $cof(\lambda)>\kappa$.
Then $j^{''}\lambda$ is unbounded in $j(\lambda)$.
\end{lem}
\begin{proof}
Take any $\alpha<\lambda$. there is some function $f:\kappa\mapsto\lambda$
and an ordinal $\beta$ such that $j(f)(\beta)=\alpha$. So take $f^{'}$
to be also $f^{'}:\kappa\mapsto\lambda$, and define it to be $f^{'}(\beta)=\cup_{\gamma<\kappa}f(\gamma)$.
Notice that since $\forall\beta<\kappa:\,f^{'}(\beta)\geq f(\beta)$
we get $j(f^{'})(\beta)\geq j(f)(\beta)$, but also since $cof(\lambda)>\kappa$
we have $\cup_{\gamma<\kappa}f(\gamma)<\lambda$ which means that
$j(\cup_{\gamma<\kappa}f(\gamma))=j(f^{'})(\beta)<j(\lambda)$ as
needed.
\end{proof}
\begin{lem}
\label{lem:union-j-assoc}Let $X=\langle x_{\delta}\mid\delta<\rho\rangle$
be an increasing sequence of ordinals with $cof(\rho)\geq\kappa^{+}$.
Then $\cup_{a\in X}j(a)=j(\cup_{a\in X}a)$.
\end{lem}
\begin{proof}
Take a subsequence of $X$ with order type $cof(\rho)$. Proving the
lemma for it is enough, so we can assume that $cof(\rho)=|\rho|$.
First, notice that $\cup_{a\in X}j(a)\leq j(\cup_{a\in X}a)$ is trivial.
For the other direction, Notice that $j(\cup_{a\in X}a)=\cup_{a\in j(X)}a$
and since $X$ is an increasing sequence so is $j(X)$. This means
that it is enough to show that $j^{''}X$ is unbounded in $j(X)$.
For that we need only to look at the indexes, and then it follows
immediately from \lemref{1}.
\end{proof}
\begin{lem}
\label{lem:M-k-sequence}$V[G\ast g]\models^{\kappa}M[G\ast g\ast H]\subset M[G\ast g\ast H]$
\end{lem}
\begin{proof}
We know $V\models^{\kappa}M\subset M$. Use \cite{key-1} (8.4) to
get $V[G\ast g]\models^{\kappa}M[G\ast g]\subset M[G\ast g]$. Now
if $x=\langle x_{\alpha}|\alpha<\kappa\rangle\in^{\kappa}M[G\ast g\ast H]$,
denote the name of $x_{\alpha}$ by $\tilde{x_{\alpha}}$. Then $\langle\tilde{x_{\alpha}}|\alpha<\kappa\rangle\in M[G\ast g]$
and therefore $\langle x_{\alpha}|\alpha<\kappa\rangle\in M[G\ast g\ast H]$.
\end{proof}
\begin{lem}
There is a generic filter for $Add(j(\kappa),j(\kappa))^{M[G\ast g\ast H]}$
over $M[G\ast g\ast H]$, and we will denote the corresponding function
by $g^{*}$.
\end{lem}
\begin{proof}
From \cite{key-1} (8.1) there is a generic filter for $Add(i(\kappa),i(\kappa))^{N[G\ast g_{0}\ast H_{0}]}$
over $N[G\ast g_{0}\ast H_{0}]$. Use \cite{key-1} to transfer it
to $M[G\ast g\ast H]$ and get a generic filter for $Add(j(\kappa),j(\kappa))^{M[G\ast g\ast H]}$.
\end{proof}
\begin{defn}
If $f$ is some partial function from $A$ to $B$, and $g^{'}$ is
some partial function from $A\times C$ to $D$, denote $g^{'}\diamond f=\{((f(\alpha),\gamma),\delta)|((\alpha,\gamma),\delta)\in g^{'},\alpha\in dom(f)\}$
(i.e., diamond maps the first argument).
\end{defn}
\begin{defn}
For a set $A$ of pairs, denote $A|_{0}=\{\alpha|\exists(\alpha,\beta)\in A\}$
and $A|_{1}=\{\beta|\exists(\alpha,\beta)\in A\}$.
\end{defn}
\begin{lem}
\label{lem:build-f}If $f\in M[G\ast g\ast H]$ is a one to one partial
function from $j(\kappa)$ to $j(\gamma)$ and $D\in M[G\ast g\ast H]$
is a maximal antichain in $Add(j(\kappa),j(\gamma))^{M[G\ast g\ast H]}$
where $\cup_{d\in D}dom(d)|_{0}\subseteq Rng(f)$, then there is some
$d\in D$ such that $d\subseteq g^{*}\diamond f$.
\end{lem}
\begin{proof}
Look at $D^{'}=\{d\diamond f^{-1}|d\in D\}\in M[G\ast g\ast H]$.
Notice that $D^{'}$ is also an antichain, since $d_{1}\diamond f^{-1},d_{2}\diamond f^{-1}$
are incompatible on $(\alpha,\beta)$ iff $d_{1},d_{2}$ are incompatible
on $(f(\alpha),\beta)$. Also notice that $D^{'}$ is maximal (for
functions with domain limited to $dom(f)\times j(\kappa)$) since
if there is another function which is incompatible with any of the
functions in $D^{'}$, applying $\diamond f$ will give a function
which is incompatible with any of the functions in $D$. So since
$dom(f)\in M[G\ast g\ast H]$ the restriction $g^{*}|_{dom(f)\times j(\kappa)}$
is generic and therefore there is some $d^{'}\diamond f^{-1}\in D^{'}$
such that $d^{'}\diamond f^{-1}\subseteq g^{*}|_{dom(f)\times j(\kappa)}\subseteq g^{*}$,
and then $d^{'}\subseteq g^{*}\diamond f$.
\end{proof}
\begin{rem}
The $j(\kappa)$-c.c. of the forcing adding $G\ast g\ast H$ to $M$
implies that every set of cardinality $j(\kappa)$ in $M[G\ast g\ast H]$
can be covered by a set in $M$ of the same cardinality. In particular
the set $\cup_{d\in D}dom(d)|_{0}$ can be covered by a set in $M$
of cardinality $j(\kappa)$ there.
\end{rem}
\begin{lem}
\label{lem:build-generic}Suppose there is some $f$ a one to one
function from $j(\kappa)$ onto $j(\gamma)$ such that for every $Y\in(P_{j(\kappa)^{+}}(j(\gamma)))^{M[G\ast g\ast H]}$
there is some $Z\subseteq j(\kappa)$ such that:
\begin{enumerate}
\item $Y\subseteq f^{''}Z$.
\item $f|_{Z}\in M[G\ast g\ast H]$.
\end{enumerate}
Then there is a generic filter for $Add(j(\kappa),j(\gamma))^{M[G\ast g\ast H]}$
over $M[G\ast g\ast H]$.
\end{lem}
\begin{proof}
Define $g^{**}=g^{*}\diamond f$. We will show that $g^{**}$ is the
function generated from a generic filter. From the definition of $g^{**}$
it is clear that it is a function $j(\gamma)\times j(\kappa)\rightarrow2$.
Let $D\in M[G\ast g\ast H]$ be a maximal antichain for $Add(j(\kappa),j(\gamma))^{M[G\ast g\ast H]}$.
Then from $j(\kappa^{+})-c.c.$ we have $|D|\leq j(\kappa)$, which
means also $|\cup_{d\in D}dom(d)|_{0}|\leq j(\kappa)$. So from the
property of $f$ there is some $Z\subseteq j(\kappa)$ such that $\cup_{d\in D}dom(d)|_{0}\subseteq f^{''}Z$,
and also $f|_{Z}\in M[G\ast g\ast H]$. Notice that this means that
$f|_{Z},D$ complete all the requirements of \lemref{build-f}, and
so there is some $d\in D$ such that $d\subseteq g^{*}\diamond f|_{Z}\subseteq g^{*}\diamond f=g^{**}$,
which proves what we wanted.
\end{proof}
\newpage{}

\section{\label{part:2}$Add(\kappa,\kappa^{+2})$}

\subsection{Building Good Models}

We will first prove the case where $\gamma=\kappa^{+2}$.
\begin{defn}
Let $\vartriangleleft$ be some well ordering on $H_{\theta}^{M[G\ast g\ast H]}$
for some big enough $\theta$, such that if $A,B\in H_{\theta}^{M[G\ast g\ast H]}$
and $|A|<|B|$ then $A\vartriangleleft B$.
\end{defn}
\begin{defn}
$X\in(P_{j(\kappa)^{+}}(H_{\theta}))^{M[G\ast g\ast H]}$ will be
called a good model if:
\begin{enumerate}
\item $M[G\ast g\ast H]\models X\preceq\langle H_{\theta}^{M[G\ast g\ast H]},\vartriangleleft\rangle$.
\item $M[G\ast g\ast H]\models|X|=j(\kappa)$.
\item \label{enu:base}$j(\kappa),j(\kappa^{+})\in X$.
\item \label{enu:ord}$X\cap j(\kappa^{+})\in j(\kappa^{+})$.
\item \label{enu:unbounded}$X\cap j^{''}\kappa^{++}$ is unbounded in $sup(X\cap j(\kappa^{++}))$.
So for every $\alpha\in X\cap j(\kappa^{++})$ there is an element
of $X\cap j^{''}\kappa^{++}$ which is $\geq\alpha$. Denote this
element by $\eta_{X,\alpha}$, and its source by $\eta_{X,\alpha}^{'}$
(i.e. $j(\eta_{X,\alpha}^{'})=\eta_{X,\alpha}\in(X\cap j^{''}\kappa^{++})\setminus\alpha$).
\item \label{enu:closed}If $D\subseteq X$ is an increasing ordinal sequence
of cofinality $>\omega$ then $\cup D\in X$.
\end{enumerate}
\end{defn}
\begin{lem}
For every $B\in(P_{j(\kappa)^{+}}(j(\kappa^{++})))^{M[G\ast g\ast H]}$
there is some good model $X_{B}$ such that $B\subseteq X_{B}$.
\end{lem}
\begin{proof}
Using Lowenheim-Skolem-Tarski build $X_{B}^{0}$ such that: 
\begin{enumerate}
\item $M[G\ast g\ast H]\models X_{B}^{0}\preceq\langle H_{\theta}^{M[G\ast g\ast H]},\vartriangleleft\rangle$.
\item $M[G\ast g\ast H]\models|X_{B}^{0}|=j(\kappa)$.
\item $j(\kappa),j(\kappa^{+})\in X_{B}^{0}$.
\item $B\subseteq X_{B}^{0}$.
\end{enumerate}
Now work inductively - assume we have $X_{B}^{i}$ which has $M[G\ast g\ast H]\models|X_{B}^{i}|=j(\kappa)$.
Then $X_{B}^{i}\cap j(\kappa^{++})$ must be bounded in $j(\kappa^{++})$,
so from \lemref{1} there is some $\eta_{i}\in j^{''}\kappa^{++}$
above it. Denote by $\langle x_{\alpha}^{i}|\alpha<\psi_{i}\rangle$
an increasing ordinal sequence which enumerates the ordinals of $X_{B}^{i}$
(notice $|\psi_{i}|=j(\kappa)$). Build $X_{B}^{i+1}$ with Lowenheim-Skolem-Tarski
such that: 
\begin{enumerate}
\item $M[G\ast g\ast H]\models X_{B}^{i+1}\preceq\langle H_{\theta}^{M[G\ast g\ast H]},\vartriangleleft\rangle$.
\item $M[G\ast g\ast H]\models|X_{B}^{i+1}|=j(\kappa)$.
\item $X_{B}^{i}\subseteq X_{B}^{i+1}$.
\item $sup(X_{B}^{i}\cap j(\kappa^{+}))\subseteq X_{B}^{i+1}$.
\item $\eta_{i}\in X_{B}^{i+1}$.
\item $\forall\beta\leq\psi_{i}:\cup_{\alpha<\beta}x_{\alpha}^{i}\in X_{B}^{i+1}$
\end{enumerate}
Finally, define $X_{B}=\cup_{i<\omega}X_{B}^{i}$. We have $B\subseteq X_{B}^{0}\subseteq X_{B}$
so it is left to check that $X_{B}$ satisfies all the properties
of a good model:
\begin{enumerate}
\item Notice that the sequence $\langle X_{B}^{i}|i<\omega\rangle$ is an
$\omega$-sequence of elements of $M[G\ast g\ast H]$, so from \lemref{M-k-sequence}
it is also in $M[G\ast g\ast H]$. Therefore so is $X_{B}$, and it
is a basic property that the union of an increasing sequence of elementary
submodels is also an elementary submodel.
\item From the induction we know that all of the $X_{B}^{i}$ are of size
$j(\kappa)$ in $M[G\ast g\ast H]$, and since there are only $\aleph_{0}$
so is their union.
\item $j(\kappa),j(k^{+})\in X_{B}^{0}\subseteq X_{B}$.
\item Let $\alpha\in X_{B}\cap j(\kappa^{+})$. Then there is some $i$
such that $\alpha\in X_{B}^{i}\cap j(\kappa^{+})$. But then $\alpha\subseteq sup(X_{B}^{i}\cap j(\kappa^{+}))\subseteq X_{B}^{i+1}\subseteq X_{B}$.
So $X_{B}\cap j(\kappa^{+})$ is closed downwards, which means that
it is an ordinal. Since $M[G\ast g\ast H]\models|X_{B}|=j(\kappa)$
that ordinal must be below $j(\kappa^{+})$ and so $X_{B}\cap j(\kappa^{+})\in j(\kappa^{+})$\@.
\item Let $\alpha\in X_{B}\cap j(\kappa^{++})$. Then there is some $i$
such that $\alpha\in X_{B}^{i}\cap j(\kappa^{++})$. But then we have
$\alpha<\eta_{i}\in X_{B}^{i+1}\cap j^{''}\kappa^{++}\subseteq X_{B}\cap j^{''}\kappa^{++}$,
which means that $X_{B}\cap j^{''}\kappa^{++}$ is unbounded in $X_{B}\cap j(\kappa^{++})$. 
\item Let $\langle d_{\alpha}|\alpha<\psi\rangle=D\subseteq X_{B}$ be some
increasing ordinal sequence of cofinality $>\omega$, i.e. $cof(\psi)>\omega$.
It is clear that it is enough to look at $D$ such that $\psi$ is
regular (otherwise just take an unbounded subsequence of $D$ of size
$cof(\psi)$), so we can assume that $\psi$ is uncountable. Define
the mapping which for $\alpha<\psi$ returns the $i$ such that $d_{\alpha}\in X_{B}^{i}$.
Notice that from $\alpha\geq\omega$ the function is regressive, so
from Fodor's lemma there is some stationary set $S^{'}\subseteq\psi$
and $i^{'}<\omega$ such that $\forall\alpha\in S^{'}:d_{\alpha}\in X_{B}^{i^{'}}$.
But then since $S^{'}$ is unbounded in $\psi$ we get that $\cup D=\cup_{\alpha\in S^{'}}d_{\alpha}$,
and we have $\cup_{\alpha\in S^{'}}d_{\alpha}\in X_{B}^{i+1}$ from
the last requirement on $X_{B}^{i+1}$. Together we have that $\cup D\in X_{B}^{i+1}\subseteq X_{B}$.
\end{enumerate}
\end{proof}

\subsection{Main Lemma}
\begin{defn}
For every ordinal $\alpha\in H_{\theta}^{M[G\ast g\ast H]}$ define
$C_{\alpha+1}=\{\alpha\}$, and for a limit ordinal $\alpha\in H_{\theta}^{M[G\ast g\ast H]}$
define $C_{\alpha}$ to be some club in $\alpha$ which is the minimal
such one under $\vartriangleleft$. Notice that from the definition
of $\vartriangleleft$ it must be of size $cof(\alpha)$.
\end{defn}
\begin{lem}
\label{lem:club}For every model $X$ such that $M[G\ast g\ast H]\models X\preceq\langle H_{\theta}^{M[G\ast g\ast H]},\vartriangleleft\rangle$
and ordinal $\alpha\in X$ we have $C_{\alpha}\in X$.
\end{lem}
\begin{proof}
First assume $\alpha$ is a successor ordinal. Then $\langle H_{\theta}^{M[G\ast g\ast H]},\vartriangleleft\rangle\models\exists\beta:\beta+1=\alpha$,
so from elementarity $X\models\exists\beta:\beta+1=\alpha$. So denote
the corresponding $\beta$ by $\alpha^{'}\in X$. Then $X\models\alpha^{'}+1=\alpha$
and from elementarity $\langle H_{\theta}^{M[G\ast g\ast H]},\vartriangleleft\rangle\models\alpha^{'}+1=\alpha$.
So $C_{\alpha}=\{\alpha^{'}\}\in X$. Assume now that $\alpha$ is
a limit ordinal. Then $\langle H_{\theta}^{M[G\ast g\ast H]},\vartriangleleft\rangle\models\text{"there is a club in \ensuremath{\alpha}"}$.
Then from elementarity we have $X\models\text{"there is a club in \ensuremath{\alpha}"}$.
Denote the minimal such one under the well ordering by $C$. But then
from elementarity it must be the minimal in $H_{\theta}^{M[G\ast g\ast H]}$,
which means it is $C_{\alpha}$ and we have $C_{\alpha}=C\in X$.
\end{proof}
\begin{lem}
There is some sequence $\langle S_{\alpha}|\alpha<\kappa^{++}\rangle=S\subseteq(P_{j(\kappa)^{+}}(H_{\theta}))^{M[G\ast g\ast H]}$
for some large enough $\theta$, which satisfies the following properties: 
\begin{enumerate}
\item $\forall B\in(P_{j(\kappa)^{+}}(j(\kappa^{++})))^{M[G\ast g\ast H]}:\:\exists A\in S:\:B\subseteq A$
.
\item If $A\in S$, $A^{*}\subseteq A\cap j(\kappa^{++})$, and for every
$a\in A^{*}$ there is some $B_{a}\in S$ indexed before $A$ and
$a\in B_{a}$, then there are $\{A_{\alpha}^{**}|\alpha<\kappa\}\subset S$
indexed before $A$ such that $A^{*}\subseteq\cup_{\alpha<\kappa}A_{\alpha}^{**}$.
\end{enumerate}
\end{lem}
\begin{proof}
All of our $S_{\alpha}$ will be good models. Since $|(P_{j(\kappa)^{+}}(j(\kappa^{++}))^{M[G\ast g\ast H]}|=\kappa^{++}$,
denote by $t^{'}\in V[G\ast g]$ a function from $\kappa^{++}$ onto
$(P_{j(\kappa)^{+}}(j(\kappa^{++}))^{M[G\ast g\ast H]}$. We will
define $S$ by induction, where in each step we add $\kappa^{+}$
elements. So assume we have done all the steps up to step $\alpha$,
and denote the sequence built so far by $S^{'}$. We want to choose
some $\langle A_{\mu}|\mu<\kappa^{+}\rangle$ as the next $\kappa^{+}$
elements of the sequence. Choose $A_{0}$ to be $X_{t^{'}(\alpha)}$(This
guarantees that after $\kappa^{++}$ steps we have the first property
of $S$). Now choose the rest of $\langle A_{\mu}|\mu<\kappa^{+}\rangle$
- assuming we have chosen up to $\mu$, choose $A_{\mu}$ to be $X_{\cup_{\mu^{'}<\mu}A_{\mu^{'}}\cup\{j(\mu)\}}$.
Notice that it is possible since for every $\mu<\kappa^{+}$ the sequence
$\langle A_{\mu^{'}}|\mu^{'}<\mu\rangle$ is of length $\leq\kappa$,
so from \lemref{M-k-sequence} it is in $M[G\ast g\ast H]$ and so
is its union. From the definition we have $A_{\mu^{'}}\subseteq A_{\mu}$
for every $\mu^{'}<\mu$, and also $j(\mu)\in A_{\mu}$. Notice that
since $\forall\mu<\kappa^{+}:\,j(\mu)\in A_{\mu}$, from \lemref{1}
$\langle A_{\mu}\cap j(\kappa^{+})|\mu<\kappa^{+}\rangle$ is unbounded
in $j(\kappa)^{+}$. Now simply add $\langle A_{\mu}|\mu<\kappa^{+}\rangle$
to the sequence built so far. 

It is left to show that the last property holds. Assume by induction
that $S^{'}$ satisfies it. Let $X$ be one of the $A_{\mu}$ and
$X^{*}\subseteq X$ as in the last property description, i.e. for
every $a\in X^{*}$ there is some $B_{a}\in S^{'}$ such that $a\in B_{a}$
(notice we ignore the case $B_{a}=A_{\mu^{'}}$ for $\mu^{'}<\mu$,
since we can just add all the $\{A_{\mu^{'}}\}_{\mu^{'}<\mu}$ to
the $A_{\alpha}^{**}$'s). Set $\alpha^{*}=\cup X^{*}$. We will show
this by induction on $\alpha^{*}$. If $\alpha^{*}$ is not a limit
ordinal, use the property for $X^{*}$ without the last element and
then add $B_{\alpha^{*}}$ to the $A_{\alpha}^{**}$'s. So assume
$\alpha^{*}$ is a limit ordinal. If $cof(\alpha^{*})\leq\kappa$
take an unbounded sequence which witnesses that, and use the property
for each of the elements of the sequence, and then just take all the
generated $A_{\alpha}^{**}$. Since each set is of size $\leq\kappa$
and we have $\leq\kappa$ such sets we are still $\leq\kappa$. So
we can assume $cof(\alpha^{*})\geq\kappa^{+}$. Notice that $M[G\ast g\ast H]\models cof(\alpha^{*})\leq j(\kappa)$
since $X\cap\alpha^{*}$ is unbounded in $\alpha^{*}$ and of size
at most $j(\kappa)$ in $M[G\ast g\ast H]$.

First, assume there is no ordinal in some submodel in $S^{'}$ that
bounds the sequence. We get: 
\[
\alpha^{*}=\cup X^{*}=\cup_{a\in X^{*}}\eta_{B_{a},a}=\cup_{a\in X^{*}}j(\eta_{B_{a},a}^{'})=j(\cup_{a\in X^{*}}\eta_{B_{a},a}^{'})
\]
 The second equality is since $\leq$ is trivial, and $\geq$ is because
if not then we have $\eta_{B_{a},a}\geq\alpha^{*}$ in contradiction
to our assumption that $\alpha^{*}$ is above all of the submodels.
The fourth equality is from \lemref{union-j-assoc} (where the $\eta_{B_{a},a}^{'}$'s
are of cofinality $\geq\kappa^{+}$ since if they aren't, so is the
cofinality of the $\eta_{B_{a},a}$'s, of which the limit is $\alpha^{*}$
which is of cofinality $\geq\kappa^{+}$). From this we get that $M[G\ast g\ast H]\models cof(\alpha^{*})\geq j(\kappa)^{+}$
which is impossible since we have seen that $M[G\ast g\ast H]\models cof(\alpha^{*})\leq j(\kappa)$.

So we can assume that the sequence is bounded by some element in a
model in $S^{'}$. Denote the minimal such element by $\alpha^{**}$
and its model by $A$. 

Assume first that $\alpha^{*}<\alpha^{**}$. From \lemref{club} we
have $C_{\alpha^{**}}\in A$. Then there is some $\rho<j(\kappa^{+})$
such that $\alpha^{*}\leq C_{\alpha^{**}}(\rho)$. Pick some $A\subseteq B\in S^{'}$
that was added with $A$ such that $\rho<B\cap j(\kappa^{+})$. But
then since $\rho,C_{\alpha^{**}}\in B$ we have $C_{\alpha^{**}}(\rho)\in B$
which is smaller than $\alpha^{**}$ and $\geq\alpha^{*}$ in contradiction
to the minimality of $\alpha^{**}$.

So $\alpha^{*}=\alpha^{**}$. Pick some $B\in S^{'}$ such that $A\subseteq B$
and $X\cap j(\kappa^{+})<B\cap j(\kappa^{+})$. Now build the following
for each $a\in X^{*}$ - start from $\beta_{0}=\alpha^{*}$. Given
$\beta_{i}>a$, take $\beta_{i+1}=min(C_{\beta_{i}}\setminus a)$.
The process stops when we reach $a$, and it must stop because otherwise
we have a set of decreasing ordinals with no minimum. Inductively
we can see that all the $\beta_{i}$ are in $X$, since if $\beta_{i},C_{\beta_{i}},a\in X$
then from elementarity $\beta_{i+1}$ is definable in $X$ (and $\beta_{0}=\alpha^{*}$
is in $X$ from property \ref{enu:closed} of good models). Notice
that inductively, each $\beta_{i}$ is in $B$ - since $\beta_{i+1}\in C_{\beta_{i}}$
and is in $X$, then $\beta_{i+1}=C_{\beta_{i}}(\alpha)$ for some
$\alpha<X\cap j(\kappa^{+})<B\cap j(\kappa^{+})$, and therefore it
is in $B$ (Notice that all the ordinals which we are looking at are
$<j(\kappa^{++})$, and so are their cofinalities). Hence $X^{*}\subseteq B$
as we wanted. 
\end{proof}
We will now build by induction a sequence $\langle f_{\alpha}|\alpha<\kappa^{++}\rangle$
of one to one partial functions from $j(\kappa)$ to $j(\kappa^{++})$
which are all compatible. The induction will also guarantee the following
properties:
\begin{enumerate}
\item $S_{\alpha}\cap j(\kappa^{++})=Rng(f_{\alpha})$.
\item $f_{\alpha}\in M[G\ast g\ast H]$.
\end{enumerate}
First, for every $\alpha<\kappa^{++}$, denote $S_{\alpha}^{'}=S_{\alpha}\cap(\cup_{\beta<\alpha}(S_{\beta}\cap j(\kappa^{++})))$
and $S_{\alpha}^{''}=(S_{\alpha}\cap j(\kappa^{++}))\setminus S_{\alpha}^{'}$.
From the last property of $S$ there are some $\{\delta_{\gamma}^{\alpha}|\gamma<\kappa\}\subseteq\alpha$
such that $S_{\alpha}^{'}\subseteq\cup_{\gamma}S_{\delta_{\gamma}^{\alpha}}$.
This means that from \lemref{M-k-sequence} we have $S_{\alpha}^{'}\in M[G\ast g\ast H]$,
which means that also $S_{\alpha}^{''}\in M[G\ast g\ast H]$. Since
the size of each $S_{\alpha}^{''}$ in $M[G\ast g\ast H]$ is at most
$j(\kappa)$ and we know $|j(\kappa)|=\kappa^{++}$, we can divide
$j(\kappa)$ into $\kappa^{++}$-many disjoint subsets, each one in
$M[G\ast g\ast H]$, where the $\alpha$-th is of size $|S_{\alpha}^{''}|_{M[G\ast g\ast H]}$,
and let $t$ be the function which describes it - i.e. $t:\kappa^{++}\mapsto P^{M[G\ast g\ast H]}(j(\kappa))$.
So assume we have built up to stage $\alpha$, and now we want to
build $f_{\alpha}$. Notice that previous $f_{\beta}$'s already define
sources for every element of $S_{\alpha}^{'}$. Since from the induction
$\forall\gamma<\kappa:\,f_{\delta_{\gamma}^{\alpha}}\in M[G\ast g\ast H]$
from \lemref{M-k-sequence} we have $f_{\alpha}^{'}=(\cup f_{\delta_{\gamma}^{\alpha}})|_{S_{\alpha}^{'}}\in M[G\ast g\ast H]$.
Define the partial function $f_{\alpha}^{''}$ to map $f_{\alpha}^{''}(t(\alpha)(\xi))$
to the $\xi$-th element of $S_{\alpha}^{''}$. By induction the sources
of $f_{\alpha}^{''}$ are all unique from previous $f_{\beta}$ since
we are always using $t$ to find new unused sources in $j(\kappa)$.
So this covers all of $S_{\alpha}^{''}$, and so if we define $f_{\alpha}=f_{\alpha}^{'}\cup f_{\alpha}^{''}\in M[G\ast g\ast H]$
it covers all of $S_{\alpha}\cap j(\kappa^{++})$. 

Finally, we can define $f=\cup_{\alpha<\kappa^{++}}f_{\alpha}$. Notice
that $f$ is a one to one function from $j(\kappa)$ onto $j(\kappa^{++})$.
From the first property of $S$ it has the property that for every
$Y\in(P_{j(\kappa)^{+}}(j(\kappa^{++})))^{M[G\ast g\ast H]}$ there
is some $Z\subseteq j(\kappa)$ such that:
\begin{enumerate}
\item $Y\subseteq f^{''}Z$.
\item $f|_{Z}\in M[G\ast g\ast H]$.
\end{enumerate}
This is exactly the property required by \lemref{build-generic},
and so we are done.

\newpage{}

\section{\label{part:3}$Add(\kappa,\kappa^{+3})$}

\subsection{Building Very Good Models}

Now we will see how to enhance the proof so it would work for $Add(\kappa,\kappa^{+3})$.
For that we will redefine $S$.
\begin{defn}
Define $I=\prod_{l\in3}(I_{l}^{0}\times I_{l}^{1})$ where $I_{l}^{0}=j^{''}\kappa^{+l+1}\setminus j(\kappa^{+l})$
and $I_{l}^{1}=j(\kappa^{+l+1})$. This means we can look at every
$i\in I$ as $i=(i(0),i(1),i(2))$. Order $I$ with co-lexicographical
order (i.e. where $l=2$ is the most significant), where each $j(\kappa^{+l+1})\setminus j(\kappa^{+l})\times j(\kappa^{+l+1})$
is ordered lexicographically. Denote this order by $\lessdot$.
\end{defn}
\begin{defn}
For $i\in I,l<3,\beta_{0}\in I_{l}^{0},\beta_{1}\in I_{l}^{1}$ define
$c(i,l,(\beta_{0},\beta_{1}))$ to be $i^{'}\in I$ which is equal
to $i$, except at coordinate $l$ where its value is changed to $(\beta_{0},\beta_{1})$.
\end{defn}
\begin{defn}
For every ordinal $\alpha\in H_{\theta}^{M[G\ast g\ast H]}$ let $t_{\alpha}$
be the isomorphism in $H_{\theta}^{M[G\ast g\ast H]}$ from $\alpha+1$
to $|\alpha|^{M[G\ast g\ast H]}$ which is minimal under $\vartriangleleft$.
\end{defn}
\begin{lem}
\label{lem:contain-t-alpha}For every model $X$ such that $M[G\ast g\ast H]\models X\preceq\langle H_{\theta}^{M[G\ast g\ast H]},\vartriangleleft\rangle$
and ordinal $\alpha\in X$ we have $t_{\alpha}\in X$.
\end{lem}
\begin{proof}
There is some $\beta$ in $X$ such that:
\[
X\models"\text{\ensuremath{\beta} is the minimal ordianl that is isomorphic \ensuremath{\alpha}}"
\]
But then from elementarity this is true in $H_{\theta}^{M[G\ast g\ast H]}$,
so $\beta=|\alpha|^{M[G\ast g\ast H]}$, and we get $|\alpha|^{M[G\ast g\ast H]}\in X$.
Notice that: 
\[
\langle H_{\theta}^{M[G\ast g\ast H]},\vartriangleleft\rangle\models"\text{there is an isomorphism from \ensuremath{\alpha} to \ensuremath{|\alpha|^{M[G\ast g\ast H]}}}"
\]
Then from elementarity we have: 
\[
X\models"\text{there is an isomorphism from \ensuremath{\alpha} to \ensuremath{|\alpha|^{M[G\ast g\ast H]}}}"
\]
Denote the minimal such isomorphism under the well ordering by $t$.
But then from elementarity it must be the minimal in $H_{\theta}^{M[G\ast g\ast H]}$,
which means it is $t_{\alpha}$ and we have $t_{\alpha}=t\in X$.
\end{proof}
\begin{defn}
For a function $t$ and a set $A$ we will define $t^{''}A=\{t(a)|a\in A\cap dom(t)\}$.
\end{defn}
\begin{rem}
We will compose functions of the form $t_{\alpha},t_{\beta}$, where
$t_{\alpha}\circ t_{\beta}=\{(a,b)|a\in dom(t_{\beta}),t_{\beta}(a)\in dom(t_{\alpha}),t_{\alpha}(t_{\beta}(a))=b\}$.
\end{rem}
\begin{defn}
\label{def:q}For $A$ an infinite set of ordinals let $q_{A}:A\rightarrow P(A)$
such that every two sets in the image are disjoint, and each is of
size $|A|$ (the existence of such mapping is immediate from $|A\times A|=|A|$).
We will also ask that $\forall\alpha\in A:\forall\beta\in q_{A}(\alpha):\alpha<\beta$,
which will be attainable for the $A$'s we will use by simply throwing
all elements of $q_{A}(\alpha)$ which are $\leq\alpha$. We will
index each set in the image by $A$ (we can do it since they have
the same cardinality).
\end{defn}
\begin{rem}
Notice that we will not ask that $q_{A}\in H_{\theta}^{M[G\ast g\ast H]}$
like we did with $t_{\alpha}$ since we will have cases where $A\notin M[G\ast g\ast H]$.
\end{rem}
\begin{defn}
\label{def:p}For every $M[G\ast g\ast H]$-cardinal $\alpha$ let
$p_{\alpha}$ be an onto function from $(\alpha^{j(\kappa)})^{M[G\ast g\ast H]}$
to $(P_{j(\kappa^{+})}\alpha)^{M[G\ast g\ast H]}$ (in $M[G\ast g\ast H]$)
such that each element in the image appears unboundedly many times.
\end{defn}
Notice that in this part we will use $p_{\alpha}$ for $\alpha$ of
the form $j(\kappa^{+l+1})$ for $l\in3$, so since we assume GCH
we get that $(\alpha^{j(\kappa)})^{M[G\ast g\ast H]}=\alpha$.

\begin{defn}
$X_{i}\in(P_{j(\kappa^{+})}(H_{\theta}))^{M[G\ast g\ast H]}$ will
be called a very good model for $i\in I$ if:
\begin{enumerate}
\item $M[G\ast g\ast H]\models X_{i}\preceq\langle H_{\theta}^{M[G\ast g\ast H]},\triangleleft\rangle$.
\item $M[G\ast g\ast H]\models|X_{i}|=j(\kappa)$.
\item \label{enu:very-good-j_k}$j(\kappa)\subseteq X_{i}$.
\item $j(\kappa),i\in X_{i}$ (from elementarity this also means that $j(\kappa^{+}),j(\kappa^{+2})\in X_{i}$).
\item \label{enu:very-good-contains-p}For every $l\in3$ if there is $\beta<i(l)_{1}$
such that $i(l)_{1}\in q_{I_{l}^{1}}(\beta)$ in the index $\beta^{'}$
(i.e. $i(l)_{1}=q_{I_{l}^{1}}(\beta)_{\beta^{'}}$), then $p_{j(\kappa^{+l+1})}(\beta^{'})\subseteq X_{i}$.
\item \label{enu:very-good-closed}$X_{i}$ is closed for ordinal sequences
of cofinalities $>\omega$, i.e. if $D\subseteq X_{i}$ is an increasing
ordinal sequence of cofinality $>\omega$ then $\cup D\in X_{i}$.
\end{enumerate}
\end{defn}
\begin{lem}
\label{lem:very-good-inter}For every $i\in I$ there is a good model
$X_{i}$ for $i$, such that for every $l\in3$ if there is $\beta_{0}<i(l)_{0}$
such that $i(l)_{0}\in q_{I_{l}^{0}}(\beta_{0})$ and $\beta_{1}<i(l)_{1}$
such that $i(l)_{1}\in q_{I_{l}^{1}}(\beta_{1})$, then $X_{c(i,l,(\beta_{0},\beta_{1}))}\subseteq X_{i}$.
\end{lem}
\begin{proof}
We will build by induction on $I$. Assume we have built up to $i$,
and now build for $i$. Using Lowenheim-Skolem-Tarski build $X_{i}^{0}$
such that:
\begin{enumerate}
\item $M[G\ast g\ast H]\models X_{i}^{0}\preceq\langle H_{\theta}^{M[G\ast g\ast H]},\triangleleft\rangle$.
\item $M[G\ast g\ast H]\models|X_{i}^{0}|=j(\kappa)$.
\item $j(\kappa)\subseteq X_{i}^{0}$.
\item $j(\kappa),i\in X_{i}^{0}$.
\item For every $l\in3$ if there is $\beta<i(l)_{1}$ such that $i(l)_{1}\in q_{I_{l}^{1}}(\beta)$
in the index $\beta^{'}$ (i.e. $i(l)_{1}=q_{I_{l}^{1}}(\beta)_{\beta^{'}}$),
then $p_{I_{l}^{1}}(\beta^{'})\subseteq X_{i}^{0}$. Notice that this
is a reasonable assumption since $i(l)_{1}$ can appear at most once
in some $q_{I_{l}^{1}}(\beta)$.
\item For every $l\in3$ if there is $\beta_{0}<i(l)_{0}$ such that $i(l)_{0}\in q_{I_{l}^{0}}(\beta_{0})$
and $\beta_{1}<i(l)_{1}$ such that $i(l)_{1}\in q_{I_{l}^{1}}(\beta_{1})$,
then $X_{c(i,l,(\beta_{0},\beta_{1}))}\subseteq X_{i}^{0}$ (It is
already built by the induction because we only made the $l$ coordinate
smaller, so $c(i,l,(\beta_{0},\beta_{1}))\lessdot i$). Notice that
this is a reasonable assumption since $i(l)_{0},i(l)_{1}$ can appear
at most once in some $q_{I_{l}^{0}}(\beta_{0}),q_{I_{l}^{1}}(\beta_{1})$
respectively.
\end{enumerate}
Now work inductively - assume we have $X_{i}^{k}$. Denote by $\langle x_{\alpha}^{k}|\alpha<\psi_{k}\rangle$
an increasing ordinal sequence which enumerates the ordinals of $X_{i}^{k}$
(notice that $|\psi_{k}|\leq j(\kappa)$). Build $X_{i}^{k+1}$ with
Lowenheim-Skolem-Tarski such that:
\begin{enumerate}
\item $M[G\ast g\ast H]\models X_{i}^{k+1}\preceq\langle H_{\theta}^{M[G\ast g\ast H]},\triangleleft\rangle$.
\item $M[G\ast g\ast H]\models|X_{i}^{k+1}|=j(\kappa)$.
\item $X_{i}^{k}\subseteq X_{i}^{k+1}$.
\item $\forall\beta\leq\psi_{k}:\cup_{\alpha<\beta}x_{\alpha}^{k}\in X_{i}^{k+1}$.
\end{enumerate}
Finally, define $X_{i}=\cup_{k<\omega}X_{i}^{k}$. Notice that the
property we asked for in the lemma is directly from $X_{i}^{0}\subseteq X_{i}$,
so it is left to check that $X_{i}$ satisfies all the properties
of a very good model:
\begin{enumerate}
\item Notice that the sequence $\langle X_{i}^{k}|k<\omega\rangle$ is an
$\omega$-sequence of elements of $M[G\ast g\ast H]$, so from \lemref{M-k-sequence}
it is also in $M[G\ast g\ast H]$. Therefore so is $X_{i}$, and it
is a basic property that the union of an increasing sequence of elementary
submodels is also an elementary submodel.
\item From the induction we know that all of the $X_{i}^{k}$ are of size
$j(\kappa)$ in $M[G\ast g\ast H]$, and since there are only $\aleph_{0}$
so is their union.
\item $j(\kappa)\subseteq X_{i}^{0}\subseteq X_{i}$.
\item $j(\kappa),i\in X_{i}^{0}\subseteq X_{i}$
\item For every $l\in3$ if there is $\beta<i(l)_{1}$ such that $i(l)_{1}\in q_{I_{l}^{1}}(\beta)$
in the index $\beta^{'}$ (i.e. $i(l)_{1}=q_{I_{l}^{1}}(\beta)_{\beta^{'}}$),
then $p_{j(\kappa^{+l+1})}(\beta^{'})\subseteq X_{i}^{0}\subseteq X_{i}$.
\item Let $\langle d_{\alpha}|\alpha<\psi\rangle=D\subseteq X_{i}$ be some
increasing ordinal sequence of cofinality $>\omega$, i.e. $cof(\psi)>\omega$.
It is clear that it is enough to look at $D$ such that $\psi$ is
regular (otherwise just take an unbounded subsequence of $D$ of size
$cof(\psi)$), so we can assume that $\psi$ is uncountable. Define
the mapping which for $\alpha<\psi$ returns the $k$ such that $d_{\alpha}\in X_{i}^{k}$.
Notice that from $\alpha\geq\omega$ the function is regressive, so
from Fodor's lemma there is some stationary set $S^{'}\subseteq\psi$
and $k^{'}<\omega$ such that $\forall\alpha\in S^{'}:d_{\alpha}\in X_{i}^{k^{'}}$.
But then since $S^{'}$ is unbounded in $\psi$ we get that $\cup D=\cup_{\alpha\in S^{'}}d_{\alpha}$,
and we have $\cup_{\alpha\in S^{'}}d_{\alpha}\in X_{i}^{k+1}$ from
the last requirement on $X_{i}^{k+1}$. Together we have that $\cup D\in X_{i}^{k+1}\subseteq X_{i}$.
\end{enumerate}
\end{proof}
\begin{lem}
\label{lem:very-good-inter-usable}Let $i\in I$ and $Q\in(P_{j(\kappa^{+})}j(\kappa^{+l+1}))^{M[G\ast g\ast H]}$
(for some $l\in3$). Then there are $\beta_{0},\beta_{1}$ where $i(l)_{0}<\beta_{0}\in I_{l}^{0}$
($\beta_{0}$ arbitrarily large) and $i(l)_{1}<\beta_{1}\in I_{l}^{1}$
such that $X_{i}\cup Q\subseteq X_{c(i,l,(\beta_{0},\beta_{1}))}$.
\end{lem}
\begin{proof}
Let $i,Q,l$ as above. There is some $i(l)_{1}<Q^{'}<j(\kappa^{+l+1})$
such that $p_{j(\kappa^{+l+1})}(Q^{'})=Q$. So choose $\beta_{0}\in q_{I_{l}^{0}}(i(l)_{0})$
(as big as we want) and $\beta_{1}=q_{I_{l}^{1}}(i(l)_{1})_{Q^{'}}$.
Notice that if we denote $i^{'}=c(i,l,(\beta_{0},\beta_{1}))$, then
$i=c(i^{'},l,i(l))$. Rewriting with $i^{'}$ we get $i^{'}(l)_{0}\in q_{I_{l}^{0}}(i(l)_{0})$
and $i^{'}(l)_{1}=q_{I_{l}^{1}}(i(l)_{1})_{Q^{'}}$, so from property
\enuref{very-good-contains-p} of very good models we get $Q=p_{j(\kappa^{+l+1})}(Q^{'})\subseteq X_{i^{'}}$.
Also from \lemref{very-good-inter} we get $X_{i}=X_{c(i^{'},l,(i(l)_{0},i(l)_{1}))}\subseteq X_{i^{'}}$.
Together we have $X_{i}\cup Q\subseteq X_{i^{'}}=X_{c(i,l,(\beta_{0},\beta_{1}))}$.
\end{proof}

\subsection{Main Lemma}
\begin{defn}
Let $S,S^{'}$ be sequences indexed by $I$, where $S_{i}=X_{i}$
and $S_{i}^{'}=S_{i}\cap(i(2)_{0}+1)\cap(t_{i(2)_{0}}^{-1})^{''}(i(1)_{0}+1)\cap(t_{i(2)_{0}}^{-1}\circ t_{i(1)_{0}}^{-1})^{''}(i(0)_{0}+1)$.
\end{defn}
\begin{lem}
Let $i\in I$ and $X^{*}\subseteq S_{i}^{'}$. Suppose that for every
$a\in X^{*}$ we have some $i_{a}\in I$ such that $i_{a}\lessdot i$
and $a\in S_{i_{a}}$. Then there are some $\langle i_{\alpha}^{'}|\alpha<\kappa\rangle$
such that $X^{*}\subseteq\cup_{\alpha<\kappa}S_{i_{\alpha}^{'}}^{'}$,
and $\forall\alpha<\kappa:i_{\alpha}^{'}\lessdot i$.
\end{lem}
\begin{proof}
Assume inductively that it is true up to $i$, and let us prove for
$i$. We will divide $X^{*}$ into three sets, and handle each one
separately - since in the end we can just take the union of the computed
indices and we will still have $\leq\kappa$ indices. We will want
to divide $X^{*}$ into sets according to the first coordinate in
which $i,i_{a}$ differ (for $a\in X^{*})$. So for $j=0,1,2$ Define
$X_{j}^{*}=\{a\in X^{*}|\forall j^{'}>j:i(j^{'})=i_{a}(j^{'})\bigwedge i_{a}(j)\neq i(j)\}$.
Notice that their union is all of $X^{*}$. First, we will look at
$X_{2}^{*}$ and prove that it can be covered by at most $\kappa$
elements of $S^{'}$ indexed before $i$. Define $\alpha_{2}^{*}=\cup X_{2}^{*}$,
and we will prove this by induction on $\alpha_{2}^{*}$. If $\alpha_{2}^{*}\in X_{2}^{*}$,
then we can throw it and prove without it, and then just add $i_{\alpha_{2}^{*}}$.
So we can assume $\alpha_{2}^{*}\notin X_{2}^{*}$, and by the induction
we can assume that $\alpha_{2}^{*}$ is a limit ordinal. If $cof(\alpha_{2}^{*})\leq\kappa$,
then take a club in $\alpha_{2}^{*}$ that witnesses this property.
Use the induction for each element of the club intersected with $X_{2}^{*}$.
Then just take the union of all the found indices, and since this
is a union of $\leq\kappa$ sets of size $\leq\kappa$, the union
size is $\leq\kappa$ as well. So assume $cof(\alpha_{2}^{*})\geq\kappa^{+}$.

First, assume $\alpha_{2}^{*}>i^{'}(2)_{0}$ for all $i^{'}\in I$
that are $\lessdot i$ and have $i^{'}(2)\neq i(2)$ (notice that
all the $i_{a}$ for $a\in X_{2}^{*}$ are such $i^{'}$). Then we
have: 
\[
\alpha_{2}^{*}=\cup X_{2}^{*}=\cup_{a\in X_{2}^{*}}a=\cup_{a\in X_{2}^{*}}i_{a}(2)_{0}=\cup_{a\in X_{2}^{*}}j(j^{-1}(i_{a}(2)_{0}))=j(\cup_{a\in X_{2}^{*}}j^{-1}(i_{a}(2)_{0}))
\]
 For the third equality $\leq$ is because $a\leq i_{a}(2)_{0}$ (since
$a\in S_{i_{a}}^{'}\subseteq(i_{a}(2)_{0}+1)$), and $\geq$ is from
our assumption. The fifth equality is from \lemref{union-j-assoc}
(where the $j^{-1}(i_{a}(2)_{0})$'s are of cofinality $\geq\kappa^{+}$
since if they aren't, so is the cofinality of the $i_{a}(2)_{0}$'s,
of which the limit is $\alpha_{2}^{*}$ which is of cofinality $\geq\kappa^{+}$).
From this we get that $M[G\ast g\ast H]\models cof(\alpha_{2}^{*})\geq j(\kappa)^{+}$
which is impossible since we have seen that $M[G\ast g\ast H]\models cof(\alpha_{2}^{*})\leq j(\kappa)$.

So we can assume that there is some such $i^{'}$ with $\alpha_{2}^{*}\leq i^{'}(2)_{0}$.
Since $i^{'}(2)_{0}\in S_{i^{'}}$ (because $i^{'}\in S_{i^{'}}$)
this means that every such $S_{i^{'}}$ contains ordinals which are
$\geq\alpha_{2}^{*}$ (and $<j(\kappa^{+3})$). Denote the minimal
possible ordinal like this which is$\geq\alpha_{2}^{*}$ by $\alpha_{2}^{**}$,
and its model by $A$ (i.e. $\alpha_{2}^{*}\leq\alpha_{2}^{**}\in A=S_{i^{'}}$).
First, assume $\alpha_{2}^{*}<\alpha_{2}^{**}$. Notice that from
\lemref{club} $C_{\alpha_{2}^{**}}\in A$, and there must be some
$\delta<j(\kappa^{+2})$ such that $C_{\alpha_{2}^{**}}(\delta)\geq\alpha_{2}^{*}$.
From \lemref{very-good-inter-usable} there must be some $\beta<j(\kappa^{+2})$
such that $\{\delta\}\cup A\subseteq X_{c(i^{'},1,\beta)}$. But then
$C_{\alpha_{2}^{**}},\delta\in X_{c(i^{'},1,\beta)}$ and so $C_{\alpha_{2}^{**}}(\delta)\in X_{c(i^{'},1,\beta)}$
which contradicts the minimality of $\alpha_{2}^{**}$. 

So we can assume $\alpha_{2}^{*}=\alpha_{2}^{**}$. From \lemref{very-good-inter-usable}
there must be some $\beta<j(\kappa^{+2})$ such that $A\cup(S_{i}\cap j(\kappa^{+2}))\subseteq S_{c(i^{'},1,\beta)}$.
Now let $a\in X_{2}^{*}$. Define $\beta_{0}=\alpha_{2}^{*}$, and
inductively $\beta_{k+1}=min(C_{\beta_{k}}\setminus a)$ as long as
possible - i.e. until we reach $a$ (Notice that it must happen after
a finite amount of steps, since otherwise we have a decreasing sequence
of ordinals with no minimum). From property \enuref{very-good-closed}
of very good models, $\beta_{0}\in S_{i}$, and then inductively from
\lemref{club} all the $\beta_{k}$ are in $S_{i}$ as well. So if
we denote $\delta_{k}<j(\kappa^{+2})$ such that $\beta_{k+1}=C_{\beta_{k}}(\delta_{k})$,
the $\delta_{k}$ must be in $S_{i}$ as well. Now look at $S_{c(i^{'},1,\beta)}$
- $\beta_{0}$ must be in it, and then inductively if $\beta_{k}$
is in it then $\beta_{k+1}=C_{\beta_{k}}(\delta_{k})$ is in as well
(the $\delta_{k}$ are in since $\delta_{k}\in S_{i}\cap j(\kappa^{+2})\subseteq S_{c(i^{'},1,\beta)}$).
So we get that $a\in S_{c(i^{'},1,\beta)}$ for every $a\in X_{2}^{*}$,
which means $X_{2}^{*}\subseteq S_{c(i^{'},1,\beta)}$.

Denote $\tilde{i}=c(i^{'},1,\beta)$ and get $X_{2}^{*}\subseteq S_{\tilde{i}}$.
If we had $X_{2}^{*}\subseteq S_{\tilde{i}}^{'}$ then we would be
done, but that is not necessarily true. Remember that $S_{\tilde{i}}^{'}=S_{\tilde{i}}\cap(\tilde{i}(2)+1)\cap(t_{\tilde{i}(2)}^{-1})^{''}(\tilde{i}(1)+1)\cap(t_{\tilde{i}(2)_{0}}^{-1}\circ t_{\tilde{i}(1)_{0}}^{-1})^{''}(\tilde{i}(0)_{0}+1)$.
We already know $X_{2}^{*}\subseteq S_{\tilde{i}}$ so it is left
to take care of $\tilde{i}(2)+1$, $(t_{\tilde{i}(2)}^{-1})^{''}(\tilde{i}(1)+1)$
and $(t_{\tilde{i}(2)_{0}}^{-1}\circ t_{\tilde{i}(1)_{0}}^{-1})^{''}(\tilde{i}(0)_{0}+1)$.
For $\tilde{i}(2)+1$ remember that $\alpha_{2}^{*}\leq i^{'}(2)=\tilde{i}(2)<\tilde{i}(2)+1$.
Since for every $a\in X_{2}^{*}$ we have $a<\alpha_{2}^{*}$ we get
$a<\tilde{i}(2)+1$ so $X_{2}^{*}\subseteq\tilde{i}(2)+1$. So it
is left to handle $(t_{\tilde{i}(2)}^{-1})^{''}(\tilde{i}(1)+1)$
and $(t_{\tilde{i}(2)_{0}}^{-1}\circ t_{\tilde{i}(1)_{0}}^{-1})^{''}(\tilde{i}(0)_{0}+1)$.
It actually might not be true that $X_{2}^{*}\subseteq(t_{\tilde{i}(2)}^{-1})^{''}(\tilde{i}(1)+1)$,
so we will need to change $\tilde{i}(1)$. Look at $(t_{\tilde{i}(2)})^{''}X_{2}^{*}$
(which is defined for every element of $X_{2}^{*}$ because $X_{2}^{*}\subseteq\tilde{i}(2)+1$).
It is a subset of $j(\kappa^{++})$ of size at most $j(\kappa)$,
so it must be bounded. So find $\beta_{0},\beta_{1}$ for $S_{\tilde{i}}$
as in \lemref{very-good-inter-usable} such that $\beta_{0}$ bounds
$(t_{\tilde{i}(2)})^{''}X_{2}^{*}$, and get $\tilde{\tilde{i}}=c(\tilde{i},1,(\beta_{0},\beta_{1}))$
such that $S_{\tilde{i}}\subseteq S_{\tilde{\tilde{i}}}$. Now if
$a\in X_{2}^{*}$ it is still in $S_{\tilde{\tilde{i}}}$ and in $\tilde{i}(2)+1=\tilde{\tilde{i}}(2)+1$,
and since $t_{\tilde{\tilde{i}}(2)}(a)=t_{\tilde{i}(2)}(a)<\beta_{0}=\tilde{\tilde{i}}(1)<\tilde{\tilde{i}}(1)+1$
we have $a\in(t_{\tilde{\tilde{i}}(2)}^{-1})^{''}(\tilde{\tilde{i}}(1)+1)$
which means $X_{2}^{*}\subseteq(t_{\tilde{\tilde{i}}(2)}^{-1})^{''}(\tilde{\tilde{i}}(1)+1)$.
Similarly we can change the 0 coordinate to get $X_{2}^{*}\subseteq(t_{\tilde{\tilde{i}}(2)_{0}}^{-1}\circ t_{\tilde{\tilde{i}}(1)_{0}}^{-1})^{''}(\tilde{\tilde{i}}(0)_{0}+1)$
Together we get $X_{2}^{*}\subseteq S_{\tilde{\tilde{i}}}\cap(\tilde{\tilde{i}}(2)+1)\cap(t_{\tilde{\tilde{i}}(2)}^{-1})^{''}(\tilde{\tilde{i}}(1)+1)\cap(t_{\tilde{\tilde{i}}(2)_{0}}^{-1}\circ t_{\tilde{\tilde{i}}(1)_{0}}^{-1})^{''}(\tilde{\tilde{i}}(0)_{0}+1)=S_{\tilde{\tilde{i}}}^{'}$.
The last thing to notice is that still $\tilde{\tilde{i}}\lessdot i$
since to get $\tilde{\tilde{i}}$ from $\tilde{i}$ and $\tilde{i}$
from $i^{'}$ we only changed the coordinates $0,1$, and we had $i^{'}\lessdot i$
because of a difference in coordinate $2$.

Now look at $X_{1}^{*}$, and we will prove that it can be covered
by at most $\kappa$ elements of $S^{'}$ indexed before $i$. Define
$\alpha_{1}^{*}=\cup t_{i(2)_{0}}^{''}X_{1}^{*}$, and we will prove
by induction on $\alpha_{1}^{*}$. If $\alpha_{1}^{*}\in t_{i(2)_{0}}^{''}X_{1}^{*}$,
then we can throw it and prove without it, and the just add $i_{\alpha_{1}^{*}}$.
So we can assume $\alpha_{1}^{*}\notin t_{i(2)_{0}}^{''}X_{1}^{*}$,
and by the induction we can assume that $\alpha_{1}^{*}$ is a limit
ordinal. If $cof(\alpha_{1}^{*})\leq\kappa$, then take a club in
$\alpha_{1}^{*}$ that witnesses this property. Use the induction
for each element of the club intersected with $t_{i(2)_{0}}^{''}X_{1}^{*}$.
Then just take the union of all the found indices, and since this
is a union of $\leq\kappa$ sets of size $\leq\kappa$, the union
size is $\leq\kappa$ as well. So assume $cof(\alpha_{1}^{*})\geq\kappa^{+}$. 

First, assume $\alpha_{1}^{*}>i^{'}(1)_{0}$ for all $i^{'}\in I$
that are $\lessdot i$ and have $i^{'}(2)=i(2),i^{'}(1)\neq i(1)$
(notice that all the $i_{a}$ for $a\in X_{1}^{*}$ are such $i^{'}$).
Then we have:
\begin{align*}
\alpha_{1}^{*} & =\cup t_{i(2)_{0}}^{''}X_{1}^{*}=\cup_{a\in X_{1}^{*}}t_{i(2)_{0}}(a)=\cup_{a\in X_{1}^{*}}i_{a}(1)_{0}=\\
 & =\cup_{a\in X_{1}^{*}}j(j^{-1}(i_{a}(1)_{0}))=j(\cup_{a\in X_{1}^{*}}j^{-1}(i_{a}(1)_{0}))
\end{align*}
For the third equality $\leq$ is because $t_{i(2)_{0}}(a)\leq i_{a}(1)_{0}$
(since $a\in S_{i_{a}}^{'}\subseteq(t_{i(2)_{0}}^{-1})^{''}(i_{a}(1)_{0}+1)$),
and $\geq$ is from our assumption. The fifth equality is from \lemref{union-j-assoc}
(where the $j^{-1}(i_{a}(1)_{0})$'s are of cofinality $\geq\kappa^{+}$
since if they aren't, so is the cofinality of the $i_{a}(1)_{0}$'s,
of which the limit is $\alpha_{1}^{*}$ which is of cofinality $\geq\kappa^{+}$).
From this we get that $M[G\ast g\ast H]\models cof(\alpha_{1}^{*})\geq j(\kappa)^{+}$
which is impossible since we have seen that $M[G\ast g\ast H]\models cof(\alpha_{1}^{*})\leq j(\kappa)$.

So we can assume that there is some such $i^{'}$ with $\alpha_{1}^{*}\leq i^{'}(1)_{0}$.
Since $i^{'}(1)_{0}\in S_{i^{'}}$ this means every such $S_{i^{'}}$
contains ordinals which are $\geq\alpha_{1}^{*}$ (and $<j(\kappa^{+2})$).
Denote the minimal possible ordinal like this which is $\geq\alpha_{1}^{*}$
by $\alpha_{1}^{**}$, and its model by $A$ (i.e. $\alpha_{1}^{*}\leq\alpha_{1}^{**}\in A=S_{i^{'}}$).
First, assume $\alpha_{1}^{*}<\alpha_{1}^{**}$. From \lemref{club}
$C_{\alpha_{1}^{**}}\in A$, and there must be some $\delta<j(\kappa^{+})$
such that $C_{\alpha_{1}^{**}}(\delta)\geq\alpha_{1}^{*}$. From \lemref{very-good-inter-usable}
there must be some $\beta<j(\kappa^{+1})$ such that $\{\delta\}\cup A\subseteq X_{c(i^{'},0,\beta)}$.
But then $C_{\alpha_{1}^{**}},\delta\in X_{c(i^{'},0,\beta)}$ and
so $C_{\alpha_{1}^{**}}(\delta)\in X_{c(i^{'},0,\beta)}$ which contradicts
the minimality of $\alpha_{1}^{**}$. 

So we can assume $\alpha_{1}^{*}=\alpha_{1}^{**}$. From \lemref{very-good-inter-usable}
there must be some $\beta$ such that $A\cup(S_{i}\cap j(\kappa^{+}))\subseteq S_{c(i^{'},0,\beta)}$.
Now let $a\in X_{1}^{*}$. Define $\beta_{0}=\alpha_{1}^{*}$, and
inductively $\beta_{k+1}=min(C_{\beta_{k}}\setminus t_{i(2)_{0}}(a))$
as long as possible - i.e. until we reach $t_{i(2)_{0}}(a)$ (Notice
that it must happen after a finite amount of steps, since otherwise
we have a decreasing sequence of ordinals with no minimum). From property
\enuref{very-good-closed} of very good models, $\beta_{0}\in S_{i}$,
and then inductively from \lemref{club} (and since $t_{i(2)_{0}}\in S_{i}$
and so $t_{i(2)_{0}}(a)\in S_{i}$) all the $\beta_{k}$ are in $S_{i}$
as well. So if we denote $\delta_{k}<j(\kappa^{+})$ such that $\beta_{k+1}=C_{\beta_{k}}(\delta_{k})$,
the $\delta_{k}$ must be in $S_{i}$ as well. Now look at $S_{c(i^{'},0,\beta)}$
- we know that $\beta_{0}\in t_{i(2)_{0}}^{''}S_{i^{'}}\subseteq t_{i(2)_{0}}^{''}S_{c(i^{'},0,\beta)}$.
Since we know $i(2)=i^{'}(2)=c(i^{'},0,\beta)(2)$, from \lemref{contain-t-alpha}
$t_{i(2)_{0}}\in S_{c(i^{'},0,\beta)}$ and therefore $\beta_{0}\in S_{c(i^{'},0,\beta)}$.
Now inductively if $\beta_{k}$ is in it then $\beta_{k+1}=C_{\beta_{k}}(\delta_{k})$
is in as well (the $\delta_{k}$ are in since $\delta_{k}\in S_{i}\cap j(\kappa^{+})\subseteq S_{c(i^{'},1,\beta)}$,
and the $C_{\beta_{k}}$ from \lemref{club}). So we get that $t_{i(2)_{0}}(a)\in S_{c(i^{'},0,\beta)}$,
and since $t_{i(2)_{0}}\in S_{c(i^{'},0,\beta)}$ we get $a\in S_{c(i^{'},0,\beta)}$.
This is true for every $a\in X_{1}^{*}$, which means $X_{1}^{*}\subseteq S_{c(i^{'},0,\beta)}$.

Denote $\tilde{i}=c(i^{'},0,\beta)$ and get $X_{1}^{*}\subseteq S_{\tilde{i}}$.
If we have $X_{1}^{*}\subseteq S_{\tilde{i}}^{'}$ then we are done.
Remember that $S_{\tilde{i}}^{'}=S_{\tilde{i}}\cap(\tilde{i}(2)_{0}+1)\cap(t_{\tilde{i}(2)_{0}}^{-1})^{''}(\tilde{i}(1)_{0}+1)\cap(t_{\tilde{i}(2)_{0}}^{-1}\circ t_{\tilde{i}(1)_{0}}^{-1})^{''}(\tilde{i}(0)_{0}+1)$.
Since $X_{1}^{*}\subseteq S_{i}^{'}$ and $i,\tilde{i}$ agree on
coordinate $2$ we get $X_{1}^{*}\subseteq(i(2)_{0}+1)=(\tilde{i}(2)_{0}+1)$.
So it is left to take care of $(t_{\tilde{i}(2)_{0}}^{-1})^{''}(\tilde{i}(1)_{0}+1)$
and $(t_{\tilde{i}(2)_{0}}^{-1}\circ t_{\tilde{i}(1)_{0}}^{-1})^{''}(\tilde{i}(0)_{0}+1)$.
Remember that $\alpha_{1}^{*}\leq i^{'}(1)_{0}=\tilde{i}(1)_{0}<\tilde{i}(1)_{0}+1$.
Since for every $a\in X_{1}^{*}$ we have $t_{i(2)_{0}}(a)<\alpha_{1}^{*}$
we get $t_{i(2)_{0}}(a)<\tilde{i}(1)_{0}+1$ so $t_{i(2)_{0}}^{''}X_{1}^{*}\subseteq(\tilde{i}(1)_{0}+1)$,
which means that $X_{1}^{*}\subseteq(t_{\tilde{i}(2)_{0}}^{-1})^{''}(\tilde{i}(1)_{0}+1)$.
To handle $(t_{\tilde{i}(2)_{0}}^{-1}\circ t_{\tilde{i}(1)_{0}}^{-1})^{''}(\tilde{i}(0)_{0}+1)$
change the 0 coordinate as we have done in the previous case ($X_{2}^{*}$),
to get $\tilde{\tilde{i}}$ with $X_{1}^{*}\subseteq S_{\tilde{\tilde{i}}}\cap(\tilde{\tilde{i}}(2)_{0}+1)\cap(t_{\tilde{\tilde{i}}(2)_{0}}^{-1})^{''}(\tilde{\tilde{i}}(1)_{0}+1)\cap(t_{\tilde{i}(2)_{0}}^{-1}\circ t_{\tilde{i}(1)_{0}}^{-1})^{''}(\tilde{i}(0)_{0}+1)=S_{\tilde{\tilde{i}}}^{'}$.
The last thing to notice is that still $\tilde{\tilde{i}}\lessdot i$
since to get $\tilde{\tilde{i}}$ from $\tilde{i}$ and $\tilde{i}$
from $i^{'}$ we only changed the coordinate $0$, and we had $i^{'}\lessdot i$
because of a difference in coordinate $1$.

Finally look at $X_{0}^{*}$, and again we will prove that it can
be covered by at most $\kappa$ elements of $S^{'}$ indexed before
$i$. Define $\alpha_{0}^{*}=\cup(t_{i(1)_{0}}\circ t_{i(2)_{0}})^{''}X_{0}^{*}$,
and we will prove by induction on $\alpha_{0}^{*}$. If $\alpha_{0}^{*}\in(t_{i(1)_{0}}\circ t_{i(2)_{0}})^{''}X_{1}^{*}$,
then we can throw it and prove without it, and then just add $i_{\alpha_{0}^{*}}$.
So we can assume $\alpha_{0}^{*}\notin(t_{i(1)_{0}}\circ t_{i(2)_{0}})^{''}X_{0}^{*}$,
and by the induction we can assume that $\alpha_{0}^{*}$ is a limit
ordinal. If $cof(\alpha_{0}^{*})\leq\kappa$, then take a club in
$\alpha_{0}^{*}$ that witnesses this property. Use the induction
for each element of the club intersected with $(t_{i(1)_{0}}\circ t_{i(2)_{0}})^{''}X_{0}^{*}$.
Then just take the union of all the found indices, and since this
is a union of $\leq\kappa$ sets of size $\leq\kappa$, the union
size is $\leq\kappa$ as well. So assume $cof(\alpha_{0}^{*})\geq\kappa^{+}$. 

First, assume $\alpha_{0}^{*}>i^{'}(0)_{0}$ for all $i^{'}\in I$
that are $\lessdot i$ and have $i^{'}(2)=i(2),i^{'}(1)=i(1),i^{'}(0)\neq i(0)$
(notice that all the $i_{a}$ for $a\in X_{0}^{*}$ are such $i^{'}$).
Then we have:
\begin{align*}
\alpha_{0}^{*} & =\cup(t_{i(1)_{0}}\circ t_{i(2)_{0}})^{''}X_{0}^{*}=\cup_{a\in X_{0}^{*}}(t_{i(1)_{0}}\circ t_{i(2)_{0}})(a)=\cup_{a\in X_{0}^{*}}i_{a}(0)_{0}=\\
 & =\cup_{a\in X_{0}^{*}}j(j^{-1}(i_{a}(0)_{0}))=j(\cup_{a\in X_{0}^{*}}j^{-1}(i_{a}(0)_{0}))
\end{align*}
For the third equality $\leq$ is because $(t_{i(1)_{0}}\circ t_{i(2)_{0}})(a)\leq i_{a}(0)_{0}$
(since $a\in S_{i_{a}}^{'}\subseteq(t_{i(2)_{0}}^{-1}\circ t_{i(1)_{0}}^{-1})^{''}(i_{a}(0)_{0}+1)$),
and $\geq$ is from our assumption. The fifth equality is from \lemref{union-j-assoc}
(where the $j^{-1}(i_{a}(0)_{0})$'s are of cofinality $\geq\kappa^{+}$
since if they aren't, so is the cofinality of the $i_{a}(0)_{0}$'s,
of which the limit is $\alpha_{0}^{*}$ which is of cofinality $\geq\kappa^{+}$).
From this we get that $M[G\ast g\ast H]\models cof(\alpha_{0}^{*})\geq j(\kappa)^{+}$
which is impossible since we have seen that $M[G\ast g\ast H]\models cof(\alpha_{0}^{*})\leq j(\kappa)$.

So we can assume that there is some such $i^{'}$ with $\alpha_{0}^{*}\leq i^{'}(0)_{0}$.
Since $i^{'}(0)_{0}\in S_{i^{'}}$ this means that every such $S_{i^{'}}$
contains ordinals which are $\geq\alpha_{0}^{*}$. Denote the minimal
possible ordinal like this which is $\geq\alpha_{0}^{*}$ by $\alpha_{0}^{**}$,
and its model by $A$ (i.e. $\alpha_{0}^{*}\leq\alpha_{0}^{**}\in A=S_{i^{'}}$).
First, assume $\alpha_{0}^{*}<\alpha_{0}^{**}$. From \lemref{club}
$C_{\alpha_{0}^{**}}\in A$, and there must be some $\delta<j(\kappa)$
such that $C_{\alpha_{0}^{**}}(\delta)\geq\alpha_{0}^{*}$. But from
property \ref{enu:very-good-j_k} of very good models $\delta\in A$,
which means $C_{\alpha_{0}^{**}}(\delta)\in A$ which contradicts
the minimality of $\alpha_{0}^{**}$. 

So we can assume $\alpha_{0}^{*}=\alpha_{0}^{**}$. Now let $a\in X_{0}^{*}$.
Define $\beta_{0}=\alpha_{0}^{*}$, and inductively $\beta_{k+1}=min(C_{\beta_{k}}\setminus(t_{i(1)_{0}}\circ t_{i(2)_{0}})(a))$
as long as possible - i.e. until we reach $(t_{i(1)_{0}}\circ t_{i(2)_{0}})(a)$
(Notice that it must happen after a finite amount of steps, since
otherwise we have a decreasing sequence of ordinals with no minimum).
From property \enuref{very-good-closed} of very good models, $\beta_{0}\in S_{i}$,
and then inductively from \lemref{club} (and since $t_{i(1)_{0}}\circ t_{i(2)_{0}}\in S_{i}$
and so $(t_{i(1)_{0}}\circ t_{i(2)_{0}})(a)\in S_{i}$) all the $\beta_{k}$
are in $S_{i}$ as well. So if we denote $\delta_{k}<j(\kappa)$ such
that $\beta_{k+1}=C_{\beta_{k}}(\delta_{k})$, the $\delta_{k}$ must
be in $S_{i}$ as well. We know that $\beta_{0}\in(t_{i(1)_{0}}\circ t_{i(2)_{0}})^{''}S_{i^{'}}$.
Since we know $i(2)=i^{'}(2)$ and $i(1)=i^{'}(1)$, from \lemref{contain-t-alpha}
$t_{i(1)_{0}}\circ t_{i(2)_{0}}\in S_{i^{'}}$ and therefore $\beta_{0}\in S_{i^{'}}$.
Now inductively if $\beta_{k}$ is in $S_{i^{'}}$ then $\beta_{k+1}=C_{\beta_{k}}(\delta_{k})$
is in as well (the $\delta_{k}$ are in from property \ref{enu:very-good-j_k}
of very good models, and the $C_{\beta_{k}}$ from \lemref{club}).
So we get that $(t_{i(1)_{0}}\circ t_{i(2)_{0}})(a)\in S_{i^{'}}$,
and since $t_{i(1)_{0}}\circ t_{i(2)_{0}}\in S_{i^{'}}$ we get $a\in S_{i^{'}}$.
This is true for every $a\in X_{0}^{*}$, so $X_{0}^{*}\subseteq S_{i^{'}}$.

If we have $X_{0}^{*}\subseteq S_{i^{'}}^{'}$ then we are done. Remember
that $S_{i^{'}}^{'}=S_{i^{'}}\cap(i^{'}(2)_{0}+1)\cap(t_{i^{'}(2)_{0}}^{-1})^{''}(i^{'}(1)_{0}+1)$.
Since $X_{0}^{*}\subseteq S_{i}^{'}$ and $i,i^{'}$ agree on coordinates
$1,2$ we get $X_{0}^{*}\subseteq(i(2)_{0}+1)\cap(t_{i(2)_{0}}^{-1})^{''}(i(1)_{0}+1)=(i^{'}(2)_{0}+1)\cap(t_{i^{'}(2)_{0}}^{-1})^{''}(i^{'}(1)_{0}+1)$.
Together we get $X_{0}^{*}\subseteq S_{i^{'}}\cap(i^{'}(2)_{0}+1)\cap(t_{i^{'}(2)_{0}}^{-1})^{''}(i^{'}(1)_{0}+1)=S_{i^{'}}^{'}$.
\end{proof}
From here the proof continues similar to the $\kappa^{+2}$ case,
using $S^{'}$ instead of $S$ (complete explanation is in the next
part).

\newpage{}

\section{\label{part:4}$Add(\kappa,\gamma)$}

\subsection{Building Great Models}

Let us see now that the result holds for $Add(\kappa,\gamma)$ for
every cardinal $\gamma$, $\kappa<\gamma<\kappa^{+\kappa}$. For that
we will again redefine $S,I$.
\begin{defn}
For a $M[G\ast g\ast H]$-cardinal $\alpha\geq j(\kappa^{+})$ define
$I_{\alpha}=I_{\alpha}^{0}\times I_{\alpha}^{1}$, where $I_{\alpha}^{1}=(\alpha^{j(\kappa)})^{M[G\ast g\ast H]}$,
$I_{\alpha}^{0}=(\alpha\cap j^{''}Ord)\setminus j(\kappa^{+})$ for
$\alpha>j(\kappa^{+})$ (e.g. $I_{j(\kappa^{++})}^{0}=j^{''}\kappa^{++}\setminus j(\kappa^{+})$),
and $I_{j(\kappa^{+})}^{0}=j^{''}\kappa^{+}$. Note that since GCH
is assumed, $(\alpha^{j(\kappa)})^{M[G\ast g\ast H]}$ is either $\alpha$
or $(\alpha^{+})^{M[G\ast g\ast H]}$. Order $I_{\alpha}$ by the
lexicographical order and denote it by $\lessdot_{\alpha}$. 
\end{defn}
\begin{lem}
$I_{\alpha}$ is well ordered.
\end{lem}
\begin{proof}
This is a basic property of lexicographical order of a finite Cartesian
product - assume there is a decreasing sequence. Look at the most
significant part. It must be decreasing as well, so it must have a
minimum, and we can look at the sequence from the first time we get
to that minimum. Now simply repeat the argument for the second coordinate.
\end{proof}
\begin{defn}
Define $I\subseteq\prod_{\alpha\in[j(\kappa^{+}),j(\gamma)]\cap Card}I_{\alpha}$,
which consists of all the functions which have a finite support (where
a coordinate $\alpha$ is ``zero'' if it has the minimal value possible
in $I_{\alpha}$). Order it co-lexicographically with $\lessdot$
(where $j(\gamma)$ is the most significant and $j(\kappa^{+})$ is
the least, i.e. $i^{'}\lessdot i$ iff there is $\alpha\in[j(\kappa^{+}),j(\gamma)]\cap Card$
such that $i^{'}(\alpha)\lessdot_{\alpha}i(\alpha)$ and for all $\alpha^{'}\in(\alpha,j(\gamma)]\cap Card$
we have $i^{'}(\alpha)=i(\alpha)$).
\end{defn}
\begin{lem}
\label{lem:I-well-ordered}$I$ is well ordered.
\end{lem}
\begin{proof}
Let $\emptyset\neq\langle a_{\alpha}|\alpha<\xi\rangle=A\subseteq I$
be a decreasing sequence. Denote by $b_{\alpha}$ the highest non
zero coordinate of $a_{\alpha}$ (which exists since every element
of $I$ has a finite support). We will proceed by induction on $b_{0}$.
Notice that $\langle b_{\alpha}|\alpha<\xi\rangle$ must be a decreasing
sequence as well. So if there is some $\alpha^{'}$ such that $b_{\alpha^{'}}<b_{0}$,
just us the induction for $\langle a_{\alpha}|\alpha^{'}\leq\alpha<\xi\rangle$.
So we can assume that for every $\alpha<\xi$ we have $b_{0}=b_{\alpha}$.
Now look at $\langle a_{\alpha}(b_{0})|\alpha<\xi\rangle$, which
must be a decreasing sequence as well. Since $I_{b_{0}}$ is well
ordered, this sequence must have a minimal value, and there is some
$\alpha^{'}$ such that for every $\alpha\in[\alpha^{'},\xi)$ the
value of $a_{\alpha}(b_{0})$ is fixed. We can throw all the elements
below $\alpha^{'}$, so assume $a_{\alpha}(b_{0})$ is fixed for every
$\alpha<\xi$. So we can now create a new sequence $\langle a_{\alpha}^{'}|\alpha<\xi\rangle$
where $a_{\alpha}^{'}$ is $a_{\alpha}$ except the value at $b_{0}$
is zeroed. So we can use the induction on this sequence, which means
that from some element this sequence has a fixed value, which means
that so did $A$.
\end{proof}
\begin{defn}
For $i\in I$ define $c(i,\alpha,\beta)$ to be $i^{'}\in I$ such
that for every index $\neq\alpha$ it is exactly like $i$, and for
$\alpha$: $i^{'}(\alpha)=\beta$. 
\end{defn}
\begin{defn}
Let $i\in I$. We will define inductively - $\gamma_{0}^{i}=j(\gamma)$.
For every $l<\omega$, assume $\gamma_{l}^{i}$ is defined, and define
$\gamma_{l+1}^{i}=|i(\gamma_{l}^{i})_{0}|^{M[G\ast g\ast H]}$ as
long as it is possible. Notice that it is a decreasing sequence of
ordinals and therefore must stop after a finite number of steps, and
it will stop when $\gamma_{l}^{i}=j(\kappa^{+})$. Denote by $l_{i}$
the length of the sequence.
\end{defn}
\begin{example}
Starting from $\gamma=\kappa^{+\omega}$ and some $i\in I$, we have
$\gamma_{0}^{i}=j(\kappa^{+\omega})$. We could have $i(j(\kappa^{+\omega}))_{0}=j(\kappa^{+10})+\omega$
which means $\gamma_{1}^{i}=j(\kappa^{+10})$. The we might have $i(j(\kappa^{+10}))_{0}=j(\kappa^{+})+2$
which means $\gamma_{2}^{i}=j(\kappa^{+})$ and $l_{i}=3$. 
\end{example}
\begin{lem}
\label{lem:cardinal-form}Every $M[G\ast g\ast H]$-cardinal $\alpha$,
$j(\kappa^{+})\leq\alpha\leq j(\gamma)$ is of the form $\alpha=j(\kappa^{+\eta})$
for some $0<\eta<\kappa$.
\end{lem}
\begin{proof}
Look at $B=[\kappa^{+},\gamma]\cap Card$. Since $\gamma<\kappa^{+\kappa}$
we have $|B|<\kappa$. Therefore $j(B)=j^{''}B$. This means that
for $\alpha$ as above there is some $\beta\in B$ such that $j(\beta)=\alpha$.
Also since $\beta\in B=[\kappa^{+},\gamma]\cap Card\subseteq[\kappa^{+},\kappa^{+\kappa})\cap Card$
there is some $0<\eta<\kappa$ such that $\beta=\kappa^{+\eta}$,
so we get $j(\kappa^{+\eta})=\alpha$ as we wanted. 
\end{proof}
\begin{lem}
\label{lem:unbounded-j}Let $\alpha$ be a $M[G\ast g\ast H]$-cardinal
such that $j(\kappa^{+})\leq\alpha\leq j(\gamma)$. Then for every
$A\subseteq\alpha$ with $|A|=j(\kappa)$ there is a subset $B\subseteq\alpha$
with $|B|<\kappa$ such that $B\subset j^{''}Ord$ and $B$ is unbounded
in $A\cup B$.
\end{lem}
\begin{proof}
First assume $M[G\ast g\ast H]\models cof(\alpha)>j(\kappa)$. So
we have $cof(j^{-1}(\alpha))>\kappa$ and from \lemref{1} we get
that $j^{''}Ord$ is unbounded inside $\alpha$. On the other hand,
since $|A|=j(\kappa)$ it is bounded in $\alpha$ - so we can just
choose the singleton of some element of $j^{''}Ord\cap\alpha$ above
$A$ as $B$ and we are done. So assume $M[G\ast g\ast H]\models cof(\alpha)\leq j(\kappa)$.
Notice that since $\alpha$ is of the form $j(\kappa^{+\beta})$ for
$1\leq\beta<\kappa$ this means that $M[G\ast g\ast H]\models cof(\alpha)<j(\kappa)$.
Using elementarity we get $cof(j^{-1}(\alpha))<\kappa$, so if we
take some increasing sequence witnessing it, and then apply $j$,
we will get the wanted $B$.
\end{proof}
\begin{defn}
For $i\in I$ and $l<l_{i}$ define $T_{l}^{i}=t_{i(\gamma_{l-1}^{i})_{0}}\circ\ldots\circ t_{i(\gamma_{0}^{i})_{0}}$
(where when $l=0$ we have $T_{0}^{i}=Id_{j(\gamma)}$).
\end{defn}
\begin{lem}
\label{lem:T-Im}$Rng(T_{l}^{i})=\gamma_{l}^{i}$.
\end{lem}
\begin{proof}
The last function in the chain is $t_{i(\gamma_{l-1}^{i})_{0}}$,
and from its definition:
\[
Rng(T_{l}^{i})=Rng(t_{i(\gamma_{l-1}^{i})_{0}})=|i(\gamma_{l-1}^{i})_{0}|^{M[G\ast g\ast H]}=\gamma_{l}^{i}
\]
\end{proof}
\begin{rem}
Note that for every $i\in I$ since $\gamma_{l_{i}-1}^{i}=j(\kappa^{+})$
we get $Rng(T_{l_{i}-1})=j(\kappa^{+})$.
\end{rem}
\begin{defn}
$X_{i}\in H_{\theta}^{M[G\ast g\ast H]}$ will be called a great model
for $i\in I$ iff:
\begin{enumerate}
\item $M[G\ast g\ast H]\models X_{i}\preceq\langle H_{\theta}^{M[G\ast g\ast H]},\vartriangleleft\rangle$.
\item $M[G\ast g\ast H]\models|X_{i}|=j(\kappa)$.
\item \label{enu:great-contains-j_k}$j(\kappa)\subseteq X_{i}$.
\item \label{enu:great-gamma}$\forall l<l_{i}:\gamma_{l}^{i}\in X_{i}$.
\item \label{enu:great-i-gamma}$\forall l<l_{i}:i(\gamma_{l}^{i})_{0}\in X_{i}$.
\item \label{enu:great-contains-p}For every $\alpha$ such that $i(\alpha)$
is not zero, if there is $\beta<i(\alpha)_{1}$ such that $i(\alpha)_{1}\in q_{I_{\alpha}^{1}}(\beta)$
in the index $\beta^{'}$ (i.e. $i(\alpha)_{1}=q_{I_{\alpha}^{1}}(\beta)_{\beta^{'}}$),
then $p_{\alpha}(\beta^{'})\subseteq X_{i}$ (where $p,q$ are as
defined in \defref{q} and \defref{p}).
\item \label{enu:great-closed-cof}If $D\subseteq X_{i}$ is an increasing
ordinal sequence of cofinality $>\omega$ then $\cup D\in X_{i}$.
\end{enumerate}
\end{defn}
\begin{lem}
\label{lem:great-inter}For every $i\in I$ there is a great model
$X_{i}$, such that for every $\alpha$ such that $i(\alpha)$ is
not zero, if there is $\beta_{0}<i(\alpha)_{0}$ such that $i(\alpha)_{0}\in q_{\alpha}(\beta_{0})$
and $\beta_{1}<i(\alpha)_{1}$ such that $i(\alpha)_{1}\in q_{(\alpha^{j(\kappa)})^{M[G\ast g\ast G]}}(\beta_{1})$,
then $X_{c(i,\alpha,(\beta_{0},\beta_{1}))}\subseteq X_{i}$ .
\end{lem}
\begin{proof}
We will build by induction on $I$ (which is possible from \lemref{I-well-ordered}).
Assume we have built up to $i$, and now build for $i$. Using Lowenheim-Skolem-Tarski
build $X_{i}^{0}$ such that: 
\begin{enumerate}
\item $M[G\ast g\ast H]\models X_{i}^{0}\preceq\langle H_{\theta}^{M[G\ast g\ast H]},\vartriangleleft\rangle$.
\item $M[G\ast g\ast H]\models|X_{i}^{0}|=j(\kappa)$.
\item $j(\kappa)\subseteq X_{i}^{0}$.
\item $\forall l<l_{i}:\gamma_{l}^{i}\in X_{i}^{0}$.
\item $\forall l<l_{i}:i(\gamma_{l}^{i})_{0}\in X_{i}^{0}$.
\item For every $\alpha$ such that $i(\alpha)$ is not zero, if there is
$\beta<i(\alpha)_{1}$ such that $i(\alpha)_{1}\in q_{I_{\alpha}^{1}}(\beta)$
in the index $\beta^{'}$ (i.e. $i(\alpha)_{1}=q_{I_{\alpha}^{1}}(\alpha)_{\beta^{'}}$),
then $p_{\alpha}(\beta^{'})\subseteq X_{i}^{0}$. Notice that this
is a reasonable assumption since there is a finite number of such
$\alpha$, and $i(\alpha)_{1}$ can appear at most once in some $q_{I_{\alpha}^{1}}(\beta)$.
\item For every $\alpha$ such that $i(\alpha)$ is not zero, if there is
$\beta_{0}<i(\alpha)_{0}$ such that $i(\alpha)_{0}\in q_{I_{\alpha}^{0}}(\beta_{0})$
and $\beta_{1}<i(\alpha)_{1}$ such that $i(\alpha)_{1}\in q_{I_{\alpha}^{1}}(\beta_{1})$,
then $X_{c(i,\alpha,(\beta_{0},\beta_{1}))}\subseteq X_{i}^{0}$ (It
is already built by the induction because we only made the $\alpha$
coordinate smaller, so $c(i,\alpha,(\beta_{0},\beta_{1}))\lessdot i$).
Notice that this is a reasonable assumption since there is a finite
number of such $\alpha$, and $i(\alpha)_{0},i(\alpha)_{1}$ can appear
at most once in some $q_{I_{\alpha}^{0}}(\beta_{0}),q_{I_{\alpha}^{1}}(\beta_{1})$
respectively.
\end{enumerate}
Now work inductively on $k<\omega$ - assume we have $X_{i}^{k}$
which has $M[G\ast g\ast H]\models|X_{i}^{k}|=j(\kappa)$. Denote
by $\langle x_{\alpha}^{k}|\alpha<\psi_{k}\rangle$ an increasing
ordinal sequence which enumerates the ordinals of $X_{i}^{k}$ (notice
that $|\psi_{k}|\leq j(\kappa)$). Build $X_{i}^{k+1}$ with Lowenheim-Skolem-Tarski
such that: 
\begin{enumerate}
\item $M[G\ast g\ast H]\models X_{i}^{k+1}\preceq\langle H_{\theta}^{M[G\ast g\ast H]},\vartriangleleft\rangle$.
\item $M[G\ast g\ast H]\models|X_{i}^{k+1}|=j(\kappa)$.
\item $X_{i}^{k}\subseteq X_{i}^{k+1}$.
\item $\forall\beta\leq\psi_{k}:\cup_{\alpha<\beta}x_{\alpha}^{k}\in X_{i}^{k+1}$.
\end{enumerate}
Finally, define $X_{i}=\cup_{k<\omega}X_{i}^{k}$. Notice that we
have the wanted property from $X_{i}^{0}\subseteq X_{i}$. It is left
to check that $X_{i}$ satisfies all the properties of a great model:
\begin{enumerate}
\item Notice that the sequence $\langle X_{i}^{k}|k<\omega\rangle$ is an
$\omega$-sequence of elements of $M[G\ast g\ast H]$, so from \lemref{M-k-sequence}
it is also in $M[G\ast g\ast H]$. Therefore so is $X_{i}$, and it
is a basic property that the union of an increasing sequence of elementary
submodels is also an elementary submodel.
\item From the induction we know that all of the $X_{i}^{k}$ are of size
$j(\kappa)$ in $M[G\ast g\ast H]$, and since there are only $\aleph_{0}$
so is their union.
\item $j(\kappa)\subseteq X_{i}^{0}\subseteq X_{i}$.
\item $\forall l<l_{i}:\gamma_{l}^{i}\in X_{i}^{0}\subseteq X_{i}$.
\item $\forall l<l_{i}:i(\gamma_{l}^{i})_{0}\in X_{i}^{0}\subseteq X_{i}$.
\item For every $\alpha$ such that $i(\alpha)$ is not zero, if there is
$\beta<i(\alpha)_{1}$ such that $i(\alpha)_{1}\in q_{I_{\alpha}^{1}}(\beta)$
in the index $\beta^{'}$ (i.e. $i(\alpha)_{1}=q_{I_{\alpha}^{1}}(\alpha)_{\beta^{'}}$),
then $p_{\alpha}(\beta^{'})\subseteq X_{i}^{0}\subseteq X_{i}$. 
\item Let $\langle d_{\alpha}|\alpha<\psi\rangle=D\subseteq X_{i}$ be some
increasing ordinal sequence of cofinality $>\omega$, i.e. $cof(\psi)>\omega$.
It is clear that it is enough to look at $D$ such that $\psi$ is
regular (otherwise just take an unbounded subsequence of $D$ of size
$cof(\psi)$), so we can assume that $\psi$ is uncountable. Define
the mapping which for $\alpha<\psi$ returns the minimal $k$ such
that $d_{\alpha}\in X_{i}^{k}$. Notice that from $\alpha\geq\omega$
the function is regressive, so from Fodor's lemma there is some stationary
set $S^{'}\subseteq\psi$ and $k^{'}<\omega$ such that $\forall\alpha\in S^{'}:d_{\alpha}\in X_{i}^{k^{'}}$.
But then since $S^{'}$ is unbounded in $\psi$ we get that $\cup D=\cup_{\alpha\in S^{'}}d_{\alpha}$,
and we have $\cup_{\alpha\in S^{'}}d_{\alpha}\in X_{i}^{k^{'}+1}$
from the last requirement on $X_{i}^{k^{'}+1}$. Together we have
that $\cup D\in X_{i}^{k^{'}+1}\subseteq X_{i}$.
\end{enumerate}
\end{proof}
\begin{lem}
\label{lem:great-inter-usable}Let $i\in I$ and $Q\in(P_{j(\kappa^{+})}\alpha)^{M[G\ast g\ast H]}$
(for some $\alpha\in dom(i)$). Then there are $\beta_{0},\beta_{1}$
where $i(\alpha)_{0}<\beta_{0}<\alpha$ ($\beta_{0}$ arbitrarily
large) and $i(\alpha)_{1}<\beta_{1}<(\alpha^{j(\kappa)})^{M[G\ast g\ast H]}$
such that $X_{i}\cup Q\subseteq X_{c(i,\alpha,(\beta_{0},\beta_{1}))}$.
\end{lem}
\begin{proof}
Let $i,Q,\alpha$ as above. There is some $i(\alpha)_{1}<Q^{'}<(\alpha^{j(\kappa)})^{M[G\ast g\ast H]}$
such that $p_{\alpha}(Q^{'})=Q$. So choose $\beta_{0}\in q_{I_{\alpha}^{0}}(i(\alpha)_{0})$
(as big as we want) and $\beta_{1}=q_{I_{\alpha}^{1}}(i(\alpha)_{1})_{Q^{'}}$.
Notice that if we denote $i^{'}=c(i,\alpha,(\beta_{0},\beta_{1}))$,
then $i=c(i^{'},\alpha,i(\alpha))$. Rewriting with $i^{'}$ we get
$i^{'}(\alpha)_{0}\in q_{\alpha}(i(\alpha)_{0})$ and $i^{'}(\alpha)_{1}=q_{I_{\alpha}^{1}}(i(\alpha)_{1})_{Q^{'}}$,
so from property \enuref{great-contains-p} of great models we get
$Q=p_{\alpha}(Q^{'})\subseteq X_{i^{'}}$. Also from \lemref{great-inter}
we get $X_{i}=X_{c(i^{'},\alpha,(i(\alpha)_{0},i(\alpha)_{1}))}\subseteq X_{i^{'}}$.
Together we have $X_{i}\cup Q\subseteq X_{i^{'}}=X_{c(i,\alpha,(\beta_{0},\beta_{1}))}$.
\end{proof}

\subsection{Main Lemma}
\begin{defn}
Let $S=\langle S_{i}|i\in I\rangle$ where for every $i\in I$, $S_{i}=X_{i}$
for some $X_{i}$ built by \lemref{great-inter}. Let $S^{'}=\langle S_{i}^{'}|i\in I\rangle$
where $S_{i}^{'}=S_{i}\cap(\cap_{l<l_{i}}((T_{l}^{i})^{-1})^{''}(i(\gamma_{l}^{i})_{0}+1))$.
\end{defn}
\begin{lem}
\label{lem:bounded-l}For every $l<l_{i}$ we have $S_{i}^{'}\subseteq dom(T_{l}^{i})$
and $(T_{l}^{i})^{''}(S_{i}^{'})\subseteq(i(\gamma_{l}^{i})_{0}+1)$.
\end{lem}
\begin{proof}
From the definition we have $S_{i}^{'}\subseteq((T_{l}^{i})^{-1})^{''}(i(\gamma_{l}^{i})_{0}+1)$
which means that $S_{i}^{'}\subseteq dom(T_{l}^{i})$ and $(T_{l}^{i})^{''}(S_{i}^{'})\subseteq(i(\gamma_{l}^{i})_{0}+1)$.
\end{proof}
\begin{lem}
\label{lem:contain-t-i-gamma}For every $i\in I$ and $l<l_{i}-1$,
we have $t_{i(\gamma_{l}^{i})_{0}}\in S_{i}$.
\end{lem}
\begin{proof}
From property \enuref{great-i-gamma} of great models we have $i(\gamma_{l}^{i})_{0}\in S_{i}$,
so we can use \lemref{contain-t-alpha} and get that $t_{i(\gamma_{l}^{i})_{0}}\in S_{i}$.
\end{proof}
\begin{lem}
\label{lem:contains-T}For every $i\in I$ and $l<l_{i}$, we have
$T_{l}^{i}\in S_{i}$.
\end{lem}
\begin{proof}
Since $T_{l}^{i}=t_{i(\gamma_{l-1}^{i})_{0}}\circ\ldots\circ t_{i(\gamma_{0}^{i})_{0}}$
simply use \lemref{contain-t-i-gamma} and get the wanted result.
\end{proof}
\begin{lem}
\label{lem:T-definable}For every $i\in I$ and $l<l_{i}$, and for
every $\alpha$, $\alpha\in S_{i}$ iff $T_{l}^{i}(\alpha)\in S_{i}$.
\end{lem}
\begin{proof}
From \lemref{contains-T} we have $T_{l}^{i}\in S_{i}$, so if $\alpha\in S_{i}$
then $T_{l}^{i}(\alpha)$ is definable in $S_{i}$ and therefore in
it, and vice versa.
\end{proof}
\begin{lem}
\label{lem:great-main}Let $i\in I$ and $X^{*}\subseteq S_{i}^{'}$.
Suppose that for every $a\in X^{*}$ we have some $i_{a}\in I$ where
$i_{a}\lessdot i$ and $a\in S_{i_{a}}^{'}$. Then there are some
$\langle i_{\alpha}^{'}|\alpha<\kappa\rangle$ such that $X^{*}\subseteq\cup_{\alpha<\kappa}S_{i_{\alpha}^{'}}^{'}$
and $\forall\alpha<\kappa:i_{\alpha}^{'}\lessdot i$.
\end{lem}
\begin{proof}
Assume by induction the claim is true up to $i$, and let us prove
it for $i$. First, for every $l<l_{i}$ define: 
\begin{align*}
I_{i,l}=\{i^{'}\in I| & \exists\alpha\in dom(i)\cap[\gamma_{l}^{i},\gamma_{l-1}^{i}):\\
 & i^{'}(\alpha)\lessdot_{\alpha}i(\alpha)\land\forall\alpha^{'}\in dom(i)\cap(\alpha,\gamma_{0}^{i}]:i(\alpha^{'})=i^{'}(\alpha^{'})\}
\end{align*}
\[
X_{l}^{*}=\{a\in X^{*}|i_{a}\in I_{i,l}\}
\]
I.e. $X_{l}^{*}$ is all the $a\in X^{*}$ such that first coordinate
(from the top) in which $i_{a}$ differ from $i$ is $\alpha$ where
$\gamma_{l}^{i}\leq\alpha<\gamma_{l-1}^{i}$ (where for $l=0$ just
ask for $\gamma_{0}^{i}\leq\alpha$). Since there are only $l_{i}<\omega$
options for $l$, we can prove the wanted result for each set separately
and then just take the union of all the results. So it is enough to
prove for one such set, and so we will assume there is some $l<\omega$
such that if the first index (from the top) in which $i_{a}$ differs
from $i$ is $\alpha$ then $\gamma_{l}^{i}\leq\alpha<\gamma_{l-1}^{i}$.
Notice that since for every $l^{'}<l$ we have $\alpha<\gamma_{l^{'}}^{i}$
we get that $T_{l}^{i}=T_{l}^{i_{a}}$. For simplicity we will denote
in this scope $T=T_{l}^{i}$. Now denote $\alpha^{*}=\cup T^{''}X^{*}$
and we will proceed by induction on $\alpha^{*}$. If $\alpha^{*}\in T^{''}X^{*}$,
then $T^{-1}(\alpha^{*})\in X^{*}$. Notice we can throw $T^{-1}(\alpha^{*})$
from $X^{*}$, prove without it, and finally add $i_{T^{-1}(\alpha^{*})}$
to the (at most) $\kappa$ indices we already have, which will still
be at most $\kappa$ indices. So assume $\alpha^{*}\notin T^{''}X^{*}$.
Notice that if $\alpha^{*}$ is a successor ordinal then $\alpha^{*}\in T^{''}X^{*}$,
which means it must be a limit ordinal. If $cof(\alpha^{*})\leq\kappa$,
we can take a club $\langle\alpha_{\beta}^{*}|\beta<cof(\alpha^{*})\rangle$
witnessing its cofinality, and then use the induction on each of $(T^{-1})^{''}((T^{''}X^{*})\cap\alpha_{\beta}^{*})$
(Notice that their union is exactly $X^{*}$), and finally take the
union of all the resulting indices. Since there are at most $\kappa$
sets of at most $\kappa$ indices, their union is of size at most
$\kappa$ as well. So we can assume that $cof(\alpha^{*})\geq\kappa^{+}$.
Notice that $M[G\ast g\ast H]\models cof(\alpha^{*})\leq j(\kappa)$
since $T^{''}S_{i}\cap\alpha^{*}$ is unbounded in $\alpha^{*}$ and
of size at most $j(\kappa)$ in $M[G\ast g\ast H]$. Now assume first
that $\alpha^{*}>i^{'}(\gamma_{l}^{i^{'}})_{0}$ for all $i^{'}\in I_{i,l}$
such that $i^{'}\lessdot i$ (notice that all the $i_{a}$ for $a\in X^{*}$
are such $i^{'}$). We get:
\begin{align*}
\alpha^{*} & =\cup T^{''}X^{*}=\cup_{a\in X^{*}}T(a)=\cup_{a\in X^{*}}i_{a}(\gamma_{l}^{i_{a}})_{0}=\\
 & =\cup_{a\in X^{*}}j(j^{-1}(i_{a}(\gamma_{l}^{i_{a}})_{0}))=j(\cup_{a\in X^{*}}j^{-1}(i_{a}(\gamma_{l}^{i_{a}})_{0}))
\end{align*}
For the third equality $\leq$ is because $T(a)\leq i_{a}(\gamma_{l}^{i_{a}})_{0}$
(since $a\in S_{i_{a}}^{'}\subseteq(T^{-1})^{''}(i_{a}(\gamma_{l}^{i_{a}})_{0}+1)$),
and $\geq$ is from our assumption. The fifth equality is from \lemref{union-j-assoc}
(where the $j^{-1}(i_{a}(\gamma_{l}^{i_{a}})_{0})$'s are of cofinality
$\geq\kappa^{+}$ since if they aren't, so is the cofinality of the
$i_{a}(\gamma_{l}^{i_{a}})_{0}$'s, of which the limit is $\alpha^{*}$
which is of cofinality $\geq\kappa^{+}$). From this we get that $M[G\ast g\ast H]\models cof(\alpha^{*})\geq j(\kappa)^{+}$
which is impossible since we have seen that $M[G\ast g\ast H]\models cof(\alpha^{*})\leq j(\kappa)$.

So we can assume that there is some such $i^{'}$ with $\alpha^{*}\leq i^{'}(\gamma_{l}^{i^{'}})_{0}$.
Since $i^{'}(\gamma_{l}^{i^{'}})_{0}\in S_{i^{'}}$ this means every
such $S_{i^{'}}$ contains ordinals which are $\geq\alpha^{*}$. Denote
the minimal possible such ordinal that is $\geq\alpha^{*}$ by $\alpha^{**}$,
and denote the model by $A=S_{i^{'}}$ (i.e. $\alpha^{**}\in A$).
Remember that $\alpha^{**}\leq i^{'}(\gamma_{l}^{i^{'}})_{0}<\gamma_{l}^{i}$. 

First, assume $\alpha^{*}<\alpha^{**}$. From \lemref{club} $C_{\alpha^{**}}\in A$.
There is some $\delta<cof(\alpha^{**})^{M[G\ast g\ast H]}$ such that
$C_{\alpha^{**}}(\delta)\geq\alpha^{*}$. If $\alpha^{**}<j(\kappa^{+})$
then $cof(\alpha^{**})^{M[G\ast g\ast H]}\leq j(\kappa)$, and then
from property \enuref{great-contains-j_k} of great models we have
$\delta\in A$ which means that $C_{\alpha^{**}}(\delta)\in A$. But
then from \lemref{T-definable} $T^{-1}(C_{\alpha^{**}}(\delta))\in A$,
which means that $\alpha^{*}\leq C_{\alpha^{**}}(\delta)\in T^{''}A$
which contradicts the minimality of $\alpha^{**}$. Otherwise $\alpha^{**}\geq j(\kappa^{+})$,
so together with $\alpha^{**}<\gamma_{l}^{i}$ and $\gamma_{l_{i}-1}=j(\kappa^{+})$,
we can assume $l<l_{i}-1$ and $j(\kappa^{+})<\gamma_{l}^{i}$. Also
$\delta<cof(\alpha^{**})^{M[G\ast g\ast H]}\leq|\alpha^{**}|^{M[G\ast g\ast H]}<\gamma_{l}^{i}$.
So notice that if we change $i^{'}$ in the coordinate $\alpha^{'}=|\alpha^{**}|^{M[G\ast g\ast H]}$
it will still be $\lessdot i$, because they differ already in a coordinate
$\geq\gamma_{l}^{i}$. From \lemref{great-inter-usable} there is
some $\beta$ such that $A,\{\delta\}\subseteq S_{c(i^{'},\alpha^{'},\beta)}$.
But then since $\delta,C_{\alpha^{**}}\in S_{c(i^{'},\alpha^{'},\beta)}$
again - $C_{\alpha^{**}}(\delta)\in S_{c(i^{'},\alpha^{'},\beta)}$,
and then also $T^{-1}(C_{\alpha^{**}}(\delta))\in S_{c(i^{'},\alpha^{'},\beta)}$,
which means that $\alpha^{*}\leq C_{\alpha^{**}}(\delta)\in T^{''}S_{c(i^{'},\alpha^{'},\beta)}$
which contradicts the minimality of $\alpha^{**}$.

So assume $\alpha^{*}=\alpha^{**}$. Again we will check first the
case where $\alpha^{*}<j(\kappa^{+})$. Then for every $a\in X^{*}$
we can define $\beta_{0}=\alpha^{*}$, and then inductively $\beta_{k+1}=min(C_{\beta_{k}}\setminus T(a))$
until $\beta_{k+1}=T(a)$ for some $k$ (it has to happen, otherwise
we have a decreasing sequence of ordinals with no minimum). Notice
that since $\beta_{0}=\alpha^{*}\in T^{''}A$ from \lemref{T-definable}
we get $\beta_{0}\in A$. All the indexes into the $C_{\beta_{k}}$
are $\leq j(\kappa)$ (since we are working on cofinalities of ordinals
$<j(\kappa^{+})$), so by induction all the $\beta_{k}$ are in $A$
and we get $T(a)\in A$. Finally, again from \lemref{T-definable},
we get $a\in A$ which is true for all $a\in X^{*}$ and finishes
the proof. So assume $j(\kappa^{+})\leq\alpha^{*}$ and denote $\alpha^{'}=|\alpha^{*}|^{M[G\ast g\ast H]}$.
Again using \lemref{great-inter-usable} there is some $\beta$ such
that $A,S_{i}\cap\alpha^{'}\subseteq S_{c(i^{'},\alpha^{'},\beta)}$.
Denote $i^{''}=c(i^{'},\alpha^{'},\beta)$ and notice that as before
$i^{''}\lessdot i$. Then for every $a\in X^{*}$ we can define $\beta_{0}=\alpha^{*}$,
and then inductively $\beta_{k+1}=min(C_{\beta_{k}}\setminus T(a))$
until $\beta_{k+1}=T(a)$ for some $k$ (it has to happen, otherwise
we have a decreasing sequence of ordinals with no minimum). Notice
that since $\beta_{0}=\alpha^{*}\in T^{''}A$ from \lemref{T-definable}
we get $\beta_{0}\in A$. Also since $X^{*}\subseteq S_{i}$ and $T$
is definable in $S_{i}$ we have $T^{''}X^{*}\subseteq S_{i}$ so
$\alpha^{*}=\cup T^{''}X^{*}$ is the limit of a sequence in $S_{i}$
and from property \enuref{great-closed-cof} of great models we get
$\beta_{0}=\alpha^{*}\in S_{i}$. Since $a\in S_{i}$ and therefore
$T(a)\in S_{i}$ we can see that inductively from \lemref{club} all
the $\beta_{k}$ are in $S_{i}$ as well. That means that if $\beta_{k+1}=C_{\beta_{k}}(\delta_{k})$
then $\delta_{k}\in S_{i}$. Notice that all the $\delta_{k}$ are
$<\alpha^{'}$ (since we are working on cofinalities of ordinals $\leq\alpha^{*}$),
so since $S_{i}\cap\alpha^{'}\subseteq S_{i^{''}}$ they are in $S_{i^{''}}$.
Again by induction since $\beta_{0}\in A\subseteq S_{i^{''}}$ all
the $\beta_{k}$ are in $S_{i^{''}}$, which means that $T(a)$ as
well. Then from \lemref{T-definable} we get $a\in S_{i^{''}}$ which
is true for all $a\in X^{*}$, i.e. $X^{*}\subseteq S_{i^{''}}$.

Remember that what we actually want is $X^{*}\subseteq S_{i^{''}}^{'}$,
but that is not necessarily true. So we will show that there are $\langle i_{\alpha}^{'}|\alpha<\kappa\rangle$,
all $\lessdot i$, which we can build from $i^{''}$ such that $X^{*}\subseteq\cup_{\alpha<\kappa}S_{i_{\alpha}^{'}}^{'}$.
Notice that since $X^{*}\subseteq S_{i}^{'}=S_{i}\cap(\cap_{l^{'}<l_{i}}((T_{l^{'}}^{i})^{-1})^{''}(i(\gamma_{l^{'}}^{i})_{0}+1))$
we have that $X^{*}\subseteq((T_{l^{'}}^{i})^{-1})^{''}(i(\gamma_{l^{'}}^{i})_{0}+1)$
for every $l^{'}<l_{i}$. Since we know the first coordinate in which
$i^{''},i$ differ is $<\gamma_{l-1}^{i}$, we have $X^{*}\subseteq((T_{l^{'}}^{i^{''}})^{-1})^{''}(i^{''}(\gamma_{l^{'}}^{i^{''}})_{0}+1)$
for every $l^{'}<l$. Also, notice that from the definition of $i^{'}$
we have $i^{'}(\gamma_{l}^{i^{'}})_{0}\geq\alpha^{*}$. Now if $a\in X^{*}$
then from the definition $T(a)\leq\alpha^{*}$, so together we have
$a\in T^{-1}(i^{'}(\gamma_{l}^{i^{'}})_{0}+1)$. Since $i^{'},i^{''}$
differ in $\alpha^{'}<\gamma_{l}^{i}$ we get $a\in((T_{l}^{i^{''}})^{-1})^{''}(i^{''}(\gamma_{l}^{i^{''}})_{0}+1)$.
So from the definition of $S_{i}^{'}$ the problem is with $X^{*}\nsubseteq((T_{l^{'}}^{i^{''}})^{-1})^{''}(i^{''}(\gamma_{l^{'}}^{i^{''}})_{0}+1)$
where $l<l^{'}<l_{i}$. So denote $\gamma^{'}$ to be the maximal
of the $\gamma_{l^{'}}^{i^{''}}$ for which $X^{*}\nsubseteq((T_{l^{'}}^{i^{''}})^{-1})^{''}(i^{''}(\gamma_{l^{'}}^{i^{''}})_{0}+1)$
(If there is no such $l^{'}$ we have $X^{*}\subseteq S_{i^{''}}^{'}$
and we are done). This means that $\gamma^{'}=\gamma_{l^{'}}^{i^{''}}$
for some $l<l^{'}<l_{i}$, i.e. $\gamma^{'}<\gamma_{l}^{i^{''}}=\gamma_{l}^{i}$.
We will continue by induction on $\gamma^{'}$. If $M[G\ast g\ast H]\models cof(\gamma^{'})>j(\kappa)$
then $(T_{l^{'}}^{i^{''}})^{''}X^{*}$ is bounded in $\gamma^{'}$
and so from \lemref{great-inter-usable} we can change $i^{''}$ on
the coordinate $\gamma^{'}$ to get $\tilde{i}$, such that $\tilde{i}(\gamma^{'})_{0}>\cup(T_{l^{'}}^{i^{''}})^{''}X^{*}$.
This will mean that $X^{*}\subseteq((T_{l^{'}}^{\tilde{i}})^{-1})^{''}(\tilde{i}(\gamma^{'})_{0}+1)$,
and we can continue from the induction on $\tilde{i}$ (notice that
$i^{''}\lessdot i$ and they differ in some coordinate $\geq\gamma_{l}^{i}$,
so changing $i^{''}$ in the coordinate $\gamma^{'}<\gamma_{l}^{i}$
to get $\tilde{i}$ still gives us $\tilde{i}\lessdot i$). So assume
$M[G\ast g\ast H]\models cof(\gamma^{'})\leq j(\kappa)$. Since $\kappa\leq j^{-1}(\gamma^{'})<\kappa^{+\kappa}$
we get $M[G\ast g\ast H]\models cof(\gamma^{'})<\kappa$. So take
some increasing sequence $\langle d_{\alpha}|\alpha<cof(\gamma^{'})\rangle\in M[G\ast g\ast H]$
which witnesses this property (i.e. $\cup d_{\alpha}=\gamma^{'}$).
Now for every $\alpha$ use \lemref{great-inter-usable} to change
$i^{''}$ on the coordinate $\gamma^{'}$ to get $i_{\alpha}^{''}$
such that $i_{\alpha}^{''}(\gamma^{'})_{0}\geq d_{\alpha}$. For each
such $i_{\alpha}^{''}$ continue by induction with $X^{*}\cap((T_{l^{'}}^{i_{\alpha}^{''}})^{-1})^{''}(i_{\alpha}^{''}(\gamma^{'})_{0}+1)$,
and finally take the union of all the results. Since this is a union
of $<\kappa$ sets of size $\leq\kappa$, the result will also be
$\leq\kappa$. Finally, notice that if $a\in X^{*}$ then there is
some $\alpha$ such that $T_{l^{'}}^{i^{''}}(a)<d_{\alpha}$ so all
of $X^{*}$ will be covered as needed.
\end{proof}
\begin{lem}
\label{lem:subset-in-s'}Let $Q\in(P_{j(\kappa^{+})}j(\gamma))^{M[G\ast g\ast H]}$.
Then there are $\langle j_{\delta}|\delta<\kappa\rangle\subseteq I$
such that $Q\subseteq\cup_{\delta<\kappa}S_{j_{\delta}}^{'}$.
\end{lem}
\begin{proof}
Start from some $i\in I$. From \lemref{great-inter-usable} we can
change $i$ and assume $Q\subseteq S_{i}$. Remember that $S_{i}^{'}=S_{i}\cap(\cap_{l<l_{i}}((T_{l}^{i})^{-1})^{''}(i(\gamma_{l}^{i})_{0}+1))$,
and denote $\alpha=max(\{\gamma_{l}^{i}|Q\nsubseteq((T_{l}^{i})^{-1})^{''}(i(\gamma_{l}^{i})_{0}+1)\}_{l<l_{i}}\})$.
Notice that if $\alpha$ is not well defined, i.e. the set is empty,
from the definition of $S_{i}^{'}$ we will get that $Q\subseteq S_{i}^{'}$
and we are done. So we can assume $\alpha$ is well defined, and we
will proceed by induction on it. Since $\alpha$ is of the form $j(\kappa^{+\eta})$
for some $0<\eta<\kappa$ (from \lemref{cardinal-form}) we have $M[G\ast g\ast H]\models cof(\alpha)<\kappa\vee cof(\alpha)>j(\kappa)$.
If $M[G\ast g\ast H]\models cof(\alpha)>j(\kappa)$ then $(T_{l}^{i})^{''}Q$
which is in $\alpha$ must be bounded, so we can use \lemref{great-inter-usable}
to change $i$ in the coordinate $\alpha$ and get $i^{'}$, such
that $Q\subseteq((T_{l}^{i^{'}})^{-1})^{''}(i^{'}(\alpha)_{0}+1)$
and $S_{i}\subseteq S_{i^{'}}$. Now we can just use the induction
for $i^{'}$. So assume $M[G\ast g\ast H]\models cof(\alpha)<\kappa$.
So an increasing sequence inside $\alpha$ with limit $\alpha$ and
of size $<\kappa$, and denote it by $\langle\alpha_{\beta}|\beta<cof(\alpha)\rangle$.
Then again from \lemref{great-inter-usable} for each $\beta<cof(\alpha)$
we can change $i$ in the coordinate $\alpha$ to get $i_{\beta}$
such that $\alpha_{\beta}<i_{\beta}(\alpha)_{0}$ and $S_{i}\subseteq S_{i_{\beta}}$.
Then continue for each $\beta$ by the induction on the set $Q\cap((T_{l}^{i})^{-1})^{''}(i_{\beta}(\alpha)_{0}+1)$,
and get some $\langle i_{\beta,\delta}|\delta<\kappa\rangle$ such
that $Q\cap((T_{l}^{i})^{-1})^{''}(i_{\beta}(\alpha)_{0}+1)\subseteq\cup_{\delta<\kappa}S_{i_{\beta,\delta}}^{'}$.
Now notice that for every $q\in Q$ since $T_{l}^{i}(q)<\alpha$ there
is some $\beta<cof(\alpha)$ such that $T_{l}^{i}(q)<\alpha_{\beta}$,
so we will have $q\in Q\cap((T_{l}^{i})^{-1})^{''}(i_{\beta}(\alpha)_{0}+1)$.
This means that $\langle i_{\beta,\delta}|\delta<\kappa,\beta<cof(\alpha)\rangle$
are the indices we looked for, and indeed since $cof(\alpha)<\kappa$
there are $\kappa$ of them.
\end{proof}
We will now build by induction a sequence $\langle f_{i}|i\in I\rangle$
of one to one partial functions from $j(\kappa)$ to $j(\gamma)$
which are all compatible. The induction will also guarantee the following
properties:
\begin{enumerate}
\item $Rng(f_{i})=S_{i}^{'}$.
\item $f_{i}\in M[G\ast g\ast H]$.
\end{enumerate}
First, for every $i\in I$, denote $A_{i}=S_{i}^{'}\cap(\cup_{i^{'}\lessdot i}S_{i^{'}}^{'})$
and $B_{i}=S_{i}^{'}\setminus A_{i}$. From \lemref{great-main} there
are some $\{i_{\alpha}^{'}|\alpha<\kappa\}$ all $\lessdot i$ such
that $A_{i}\subseteq\cup_{\alpha}S_{i_{\alpha}^{'}}^{'}$. This means
that from \lemref{M-k-sequence} we have $A_{i}\in M[G\ast g\ast H]$,
which means that also $B_{i}\in M[G\ast g\ast H]$. Since the size
of each $B_{i}$ in $M[G\ast g\ast H]$ is at most $j(\kappa)$ and
we know $|I|=\gamma=|j(\kappa)|$, we can divide $j(\kappa)$ into
$|I|$-many disjoint subsets, each one in $M[G\ast g\ast H]$, where
the $i$-th is of size $|B_{i}|_{M[G\ast g\ast H]}$, and let $t$
be the function which describes it - i.e. $t:I\mapsto P^{M[G\ast g\ast H]}(j(\kappa))$.
So assume we have built up to stage $i$, and now we want to build
$f_{i}$. Notice that previous $f_{i^{'}}$'s already define sources
for every element of $A_{i}$. Since from the induction $\forall\alpha<\kappa:\,f_{i_{\alpha}^{'}}\in M[G\ast g\ast H]$
from \lemref{M-k-sequence} we have $f_{i}^{'}=(\cup_{\alpha}f_{i_{\alpha}^{'}})|_{A_{i}}\in M[G\ast g\ast H]$.
Define the partial function $f_{i}^{''}$ to map $f_{i}^{''}(t(i)(\xi))$
to the $\xi$-th element of $B_{i}$. By induction the sources of
$f_{i}^{''}$ are all unique from previous $f_{i^{'}}$ since we are
always using $t$ to find new unused sources in $j(\kappa)$. So this
covers all of $B_{i}$, and so if we define $f_{i}=f_{i}^{'}\cup f_{i}^{''}\in M[G\ast g\ast H]$
it covers all of $S_{i}^{'}$. 

Finally, we can define $f=\cup_{i\in I}f_{i}$. Notice that $f$ is
a one to one function from $j(\kappa)$ onto $j(\gamma)$. If $Y\in(P_{j(\kappa)^{+}}(j(\gamma)))^{M[G\ast g\ast H]}$
then from \lemref{subset-in-s'} there are $\langle j_{\delta}|\delta<\kappa\rangle\subseteq I$
such that $Y\subseteq\cup_{\delta<\kappa}S_{j_{\delta}}^{'}$. Then
$Z=\cup_{\delta<\kappa}dom(f_{j_{\delta}})$ is in $M[G\ast g\ast H]$
and $Y\subseteq f^{''}Z$ so we can use \lemref{build-generic} and
we are done.

\end{document}